\documentclass[12pt]{amsart}
\usepackage{txfonts}
\usepackage{amscd,amssymb,mathabx,subfigure}
\usepackage[all]{xy}
\usepackage{graphicx,calrsfs}
\usepackage[colorlinks,plainpages,backref,urlcolor=blue]{hyperref}
\usepackage{tikz}
\usetikzlibrary{calc}
\usepackage{mdwlist}

\newtheorem{theorem}{Theorem}[section]

\newtheorem{lemma}[theorem]{Lemma}
\newtheorem{prop}[theorem]{Proposition}

\theoremstyle{definition}

\newtheorem{example}[theorem]{Example}
\newtheorem{remark}[theorem]{Remark}

\newcommand{\Z}{\mathbb{Z}}
\newcommand{\Q}{\mathbb{Q}}

\newcommand{\C}{\mathbb{C}}

\renewcommand{\L}{\mathbb{L}}

\newcommand{\PP}{\mathbb{P}}

\renewcommand{\k}{\Bbbk}
\newcommand{\GG}{\mathbb{G}}

\DeclareMathAlphabet{\pazocal}{OMS}{zplm}{m}{n}
\newcommand{\A}{{\pazocal{A}}}

\newcommand{\cC}{{\pazocal{C}}}
\newcommand{\OO}{{\pazocal O}}

\newcommand{\RR}{{\mathcal R}}
\newcommand{\VV}{{\mathcal V}}

\newcommand{\F}{{\mathcal{F}}}

\newcommand{\cE}{{\mathcal{E}}}

\newcommand{\G}{\Gamma}
\newcommand{\sv}{\sf{V}}
\newcommand{\se}{\sf{E}}

\newcommand{\g}{{\mathfrak{g}}}
\newcommand{\h}{{\mathfrak{h}}}
\newcommand{\gl}{{\mathfrak{gl}}}
\renewcommand{\sl}{{\mathfrak{sl}}}
\newcommand{\sol}{{\mathfrak{sol}}}
\newcommand{\m}{{\mathfrak{m}}}

\DeclareMathOperator{\rank}{rank}

\DeclareMathOperator{\im}{im}

\DeclareMathOperator{\id}{id}

\DeclareMathOperator{\GL}{GL}
\DeclareMathOperator{\SL}{SL}
\DeclareMathOperator{\PSL}{PSL}

\DeclareMathOperator{\Hom}{{Hom}}
\DeclareMathOperator{\spn}{span}

\DeclareMathOperator{\proj}{pr}

\DeclareMathOperator{\OS}{OS}

\DeclareMathOperator{\Lie}{Lie}

\newcommand{\surj}{\twoheadrightarrow}
\newcommand{\inj}{\hookrightarrow}

\def\dot{\mathchar"013A}
\newcommand{\hdot}{{\raise1pt\hbox to0.35em{\Huge $\dot$}}}

\newcommand{\cdga}{\ensuremath{\mathsf{cdga}}}

\newcommand{\lcs}{\ensuremath{\mathsf{lcs}}}

\definecolor{dkgreen}{RGB}{0,100,0}
\definecolor{dkbrown}{RGB}{139,69,19}

\begin{document}
\date{February 20, 2017}

\title[Partial configuration spaces]{%
On the geometry and topology of partial configuration spaces of Riemann surfaces}

\author[B.~Berceanu]{Barbu Berceanu}
\address{Simion Stoilow Institute of Mathematics,
P.O. Box 1-764, RO-014700 Bucharest, Romania}
\email{Barbu.Berceanu@imar.ro}

\author[A.D.~M\u acinic]{Daniela Anca~M\u acinic}
\address{Simion Stoilow Institute of Mathematics,
P.O. Box 1-764, RO-014700 Bucharest, Romania}
\email{Anca.Macinic@imar.ro}

\author[\c S.~Papadima]{\c Stefan Papadima$^1$}
\address{Simion Stoilow Institute of Mathematics,
P.O. Box 1-764,
RO-014700 Bucharest, Romania}
\email{Stefan.Papadima@imar.ro}
\thanks{$^1$Partially supported by  PN-II-ID-PCE-2011-3-0288, grant 132/05.10.2011}

\author[C.R.~Popescu]{Clement Radu Popescu$^2$}
\address{Simion Stoilow Institute of Mathematics, 
P.O. Box 1-764, RO-014700 Bucharest, Romania}
\email{Radu.Popescu@imar.ro}
\thanks{$^2$Supported by a grant of the Romanian
National Authority for Scientific Research, CNCS-UEFISCDI,
project number PN-II-RU-TE-2012-3-0492}



\begin{abstract}
We examine complements (inside products of a smooth projective complex curve
of arbitrary genus) of unions of diagonals indexed by the edges of an arbitrary simple  
graph. We use Orlik--Solomon models associated to these quasi-projective manifolds to compute
pairs of analytic germs at the origin, both for rank 1 and 2 representation varieties of
their fundamental groups, and for degree 1 topological Green--Lazarsfeld loci. As a
corollary, we describe all regular surjections with connected generic fiber, defined on
the above complements onto smooth complex curves of negative Euler characteristic. 
We show that the nontrivial part at the origin, for both rank 2 representation varieties 
and their degree 1 jump loci, comes from curves of general type, via the above regular maps.
We compute explicit finite presentations for the Malcev Lie algebras of the fundamental groups, 
and we analyze their formality properties. 
\end{abstract}

\maketitle
\setcounter{tocdepth}{1}
\tableofcontents

\section{Introduction and statement of results}
\label{sec:intro}

Let $\G$ be a finite simple graph with cardinality $n$ vertex set $\sv$ and edge set $\se$.
The {\em partial configuration space} of type $\G$ on a space $\Sigma$ is
\begin{equation}
\label{eq:defconf}
F(\Sigma,\G)=\{z\in\Sigma^{\sv}\,\mid\,z_i\ne z_j,\quad\text{for all}\quad ij\in\se\}.
\end{equation}
When $\G=K_n$, the complete graph with $n$ vertices, $F(\Sigma, \G)$ is the classical ordered
configuration space of $n$ distinct points in $\Sigma$. In this note, we analyze the interplay
between geometry and topology, when $\Sigma= \Sigma_g$ is a compact genus $g$
Riemann surface, with partial configuration space denoted $F(g, \G)$, with special emphasis
on fundamental groups. The {\em partial pure braid groups} of type $\G$, in genus $g$,
$P(g, \G)=\pi_1(F(g, \G))$, are natural generalizations of classical pure braid groups,
which correspond to the case when $\G=K_n$ and $\Sigma = \C$. When the graph is not complete,
the classical approach to pure braid groups based on Fadell--Neuwirth fibrations does not work in full
generality. Nevertheless, we are able in this note to compute rather delicate invariants of 
arbitrary partial pure braid groups, using techniques developed in \cite{DP-ccm, MPPS}.

Viewing $\Sigma_g$ as a smooth genus $g$ complex projective curve, $F(g, \G)$
acquires the structure of an irreducible, smooth, quasi-projective complex variety
(for short, a {\em quasi-projective manifold}). For such a quasi-projective manifold $M$,
important geometric information is provided by maps onto manifolds of smaller dimension.
Particularly interesting are the {\em admissible maps} in the sense of Arapura \cite{A}, i.e.,
the regular surjections onto quasi-projective curves, $f: M\to S$, having connected generic
fiber. The admissible map $f$ is called of general type if $\chi (S)<0$. We know from 
\cite{A} that the set of admissible maps of general type on $M$, modulo reparametrization 
at the target, denoted $\cE (M)$, is finite and is intimately related to the so-called 
{\em cohomology jump loci} of $\pi:=\pi_1(M)$.

When $M=F(g, \G)$, it is relatively easy to construct certain admissible maps of general 
type on $M$, associated to complete graphs  embedded in $\G$, $f: K_m \inj \G$; see 
Section \ref{sec:pencils}. For $g\ge 2$, the relevant $m$ equals 1 and 
$f_i:  F(g, \G) \to \Sigma_g$ is induced by the projection specified by the corresponding 
vertex $i\in \sv$. For $g=1$, the relevant $m$ is 2 and
$f_{ij}:  F(1, \G) \to \Sigma_1 \setminus \{ 0\}$ is given by the projection
corresponding to $ij\in \se$, followed by the difference map on the elliptic curve 
$\Sigma_1$. For $g=0$, the relevant $m$ equals 4 and 
$f_{ijkl}:  F(0, \G) \to \PP^1 \setminus \{ 0, 1, \infty\}$ 
is  the composition of the cross-ratio with the projection associated to the vertex set of 
the embedded $K_4$. Our first main result, proved in Section \ref{sec:pencils},
establishes that there are no other admissible maps of general type on $M=F(g, \G)$.

\begin{theorem}
\label{thm:main1}
A complete set of representatives for $\cE (F(g, \G))$ is given by the admissible maps 
of general type described above.
\end{theorem}

A basic topological invariant of a connected finite CW-complex $M$ related to its cohomology
jump loci is the {\em Malcev Lie algebra} of the fundamental group $\pi:=\pi_1(M)$, 
cf. \cite{DP-ccm}. The Malcev Lie algebra $\m (\pi)$ of a group, over a characteristic zero
field $\k$, defined by Quillen in \cite{Q}, is a complete $\k$--Lie algebra, whose filtration
satisfies certain axioms, obtained by taking the primitives in the completion of the group 
ring $\k \pi$ with respect to the powers of the augmentation ideal.

Following Sullivan \cite{S}, we will say that a finitely generated group $\pi$ is 
{\em $1$--formal} if its Malcev Lie algebra is isomorphic to the completion with respect to 
the lower central series ($\lcs$) filtration of a quadratic Lie algebra $L$ (i.e., a Lie
algebra presented by degree $1$ generators and relations of degree $2$): 
$\m(\pi)\simeq \widehat{L}$. $1$--formal groups enjoy many pleasant topological properties, 
see for instance \cite{DPS09}. The $1$--formality of classical pure braid groups and pure
welded braid groups also has strong consequences in the corresponding theories
of finite type invariants, as shown in \cite{BP}.

In Section \ref{sec:mal}, we compute the Malcev Lie algebras of partial pure braid groups 
and determine precisely when they are  $1$--formal, as follows. Our next main result extends
computations done by Bezrukavnikov \cite{B} (for $g\geq 1$ and $\G=K_n$)
and Bibby--Hilburn \cite{BH} (for $g\geq 1$ and chordal graphs). Moreover, in our
presentations below redundant relations have been eliminated, for $g\geq 1$.

\begin{theorem}
\label{thm:main2}
The Malcev Lie algebra $\m (P(g, \G))$ is isomorphic to the $\lcs$ completion of a finitely
presented Lie algebra, $L(g, \G)$, with generators in degree $1$ and relations in degrees 
$2$ and $3$, described in Proposition \ref{prop:mal0} for $g=0$ and Proposition
\ref{prop:malpos} for $g\geq 1$. The group $P(g, \G)$ is not $1$--formal if and only if 
$g=1$ and the graph $\G$ contains a $K_3$ subgraph.
\end{theorem}

Now, we move to our unifying theme: the interplay between the geometry of a  quasi-projective
manifold $M$, encoded by a smooth compactification $\overline{M}$, and the embedded 
topological jump loci of $M$. We start by recalling a couple of relevant definitions and 
facts related to the topological side of this story. Fix  $q\in \Z_{>0} \cup \{ \infty \}$.
We will say that $M$ is a $q$--finite space if (up to homotopy) $M$ is a connected CW-complex
with finite $q$--skeleton, whose (finitely generated) fundamental group will be denoted by 
$\pi$. Let $\iota: \GG \to \GL (V)$ be a morphism of complex linear algebraic groups. The
associated {\em characteristic varieties} (in degree $i\ge 0$ and depth $r\ge 0$),
\begin{equation}
\label{eq:defv}
\VV^i_r(M, \iota)=\{\rho \in \Hom(\pi, \GG) \mid
\dim H^i(M, {}_{\iota\rho}V)\ge r\} \, ,
\end{equation}
are Zariski closed subvarieties (for $i\le q$) of the affine {\em representation variety} 
$\Hom(\pi, \GG)$, for which the trivial representation provides a natural basepoint, 
$1\in \Hom(\pi, \GG)$. These cohomology jump loci are called
{\em topological Green--Lazarsfeld loci} for $r=1$. They were introduced in the rank one 
case (i.e., for $\iota=\id_{\C^{\times}}$) in \cite{GL}, for a smooth projective complex
variety $M$. In the rank one case, we simplify notation to $\VV^i_r(M)$. Note that, 
in general, $\VV^1_r(M, \iota):=\VV^1_r(\pi, \iota)$ depends only on $\pi$, for all $r$.

We go on by describing the infinitesimal analogs of the above notions, following \cite{DP-ccm}.
Let $(A^{\hdot}, d)$ be a complex Commutative Differential Graded Algebra 
with positive grading (for short, a 
$\cdga$). We will say that $A^{\hdot}$ is  $q$--finite if $A^0=\C \cdot 1$ and 
$\sum_{i=1}^q \dim A^i < \infty$. Let $\theta: \g \to \gl (V)$ be a finite-dimensional
representation of a finite-dimensional complex Lie algebra. The affine variety of 
{\em flat connections}, $\F(A, \g)$, consists of the solutions in $A^1\otimes \g$ of the
Maurer--Cartan equation, has the trivial flat connection $0$ as a natural basepoint,
and is natural in both $A$ and $\g$. For $\omega \in \F(A, \g)$, there is an associated
covariant derivative, $d_{\omega}: A^{\hdot}\otimes V \to A^{\hdot +1}\otimes V$,
with $d_{\omega}^2=0$, by flatness. The {\em resonance varieties}
\begin{equation}
\label{eq:defr}
\RR^i_r(A, \theta)= \{\omega \in \F(A,\g)
\mid \dim H^i(A \otimes V, d_{\omega}) \ge  r\}
\end{equation}
are Zariski closed subvarieties (for $i\le q$). We use the simplified notation $\RR^i_r(A)$ 
in the rank one case (i.e., when $\theta=\id_{\C}$).

We say that the $\cdga$ $A^{\hdot}$ is  a $q$--model of $M$ (and omit $q$ from all terminology
when $q=\infty$) if $A^{\hdot}$ has the same Sullivan $q$--minimal model as the DeRham 
$\cdga$ $\Omega^{\hdot}(M)$, cf. \cite{S}. In particular, $H^{\hdot}(A)\simeq H^{\hdot}(M)$, 
as graded algebras, when $A$ is a model of $M$. 

The link between topological and infinitesimal objects is provided by Theorem B from 
\cite{DP-ccm}. Assume that both $A$ and $M$ are $q$--finite and $A$ is  a $q$--model of $M$.
Denote by $\theta$ the tangential representation of $\iota$. Then, for $i\le q$ and $r\ge 0$,
the embedded analytic germs at $1$, $\VV^i_r(M, \iota)_{(1)} \subseteq  \Hom(\pi, \GG)_{(1)}$,
are isomorphic to the corresponding embedded germs at $0$, 
$\RR^i_r(A, \theta)_{(0)} \subseteq \F(A,\g)_{(0)}$. Moreover, by Theorem A from \cite{DP-ccm},
if $\pi$ is a finitely generated group then the germ $\Hom(\pi, \GG)_{(1)}$ depends only on 
the Malcev Lie algebra $\m (\pi)$ and the Lie algebra of $\GG$. 

Finally, assume that $M$ is a quasi-projective manifold and $M=\overline{M} \setminus D$ is 
a smooth compactification obtained by adding at infinity a hypersurface arrangement $D$ in 
$\overline{M}$ (in the sense of Dupont \cite{D}). Then there is an associated (natural, finite)
{\em Orlik--Solomon model} $A^{\hdot}(\overline{M}, D)$ of the finite space $M$, constructed in
\cite{D}. It follows from Theorem C in \cite{DP-ccm} that
this model $A$ determines $\cE (M)$, which is in bijection with the positive-dimensional
irreducible components through the origin, for both $\RR^1_1(A)$ and $\VV^1_1(M)$. 

When $M=F(g, \G)$, we may take $\overline{M}=\Sigma_g^{\sv}$ and 
$D_{\G}=\bigcup_{ij \in \se} \Delta_{ij}$ (the union of the diagonals associated to the 
edges of the graph). We prove Theorem \ref{thm:main1} by computing the irreducible
decomposition of $\RR^1_1(A)$, for the Orlik--Solomon model $A=A(\overline{M},D_{\G})$. When $g=1$ 
and $\G=K_n$, the result follows from a more precise description of all positive-dimensional
components of $\VV^1_1(M)$, obtained by Dimca in \cite{D10}. Given a $1$--finite $1$--model $A$ of a
connected CW-space $M$, we show in Theorem \ref{thm:malhol} that the Malcev Lie algebra 
$\m (\pi_1(M))$ is isomorphic to the $\lcs$ completion of the {\em holonomy Lie algebra} of $A$, 
introduced in \cite{MPPS}. This general result is the basic tool for the proof of Theorem \ref{thm:main2},
where $M=F(g, \G)$ and $A=A(\overline{M},D_{\G})$.

$\SL_2(\C)$--representation varieties received a lot of attention, both in topology and
algebraic geometry. In order to describe their germs at $1$
for partial pure braid groups, together with the embedded germs of associated non-abelian
characteristic varieties (in degree $1$ and depth $1$),
we use their infinitesimal analogs, described above. Let  $\theta : \g \to \gl (V)$ be 
a finite-dimensional representation of $\g=\sl_2$ or $\sol_2$,
the Lie algebra of $\SL_2(\C)$ or of its standard Borel subgroup. To state our next main
result, we need two definitions from \cite{MPPS}. Denote by 
$\F^1(A,\g)\subseteq \F(A,\g)$ the flat connections of the form $\omega=\eta \otimes g$, with
$d\eta=0$ and $g\in \g$, and set 
$\Pi (A, \theta)= \{ \omega\in \F^1(A,\g)\, \mid \, \det \theta (g)=0 \}$. To have a uniform
notation, denote by $f: F(g, \G) \to S=\overline{S} \setminus F$ the admissible maps from
Theorem \ref{thm:main1}, where $\overline{S} =\Sigma_g$ and $F\subseteq \overline{S}$ 
is a finite subset (in particular, a hypersurface arrangement in $\overline{S}$). 
To avoid trivialities, we will assume in genus $0$ that $H^1(F(g, \G))\ne 0$. 
(The complete description of $H^1(F(g, \G))$ may be found in Lemma \ref{lem:prel}; what 
happens in general with the embedded topological Green--Lazarsfeld loci in degree $1$ of $M$ 
at the origin, when $b_1(M)=0$, is explained in Section \ref{sec:sl2}.)

\begin{theorem}
\label{thm:main3}
In the above setup, there is a regular extension of $f$, $\overline{f} : (\overline{M},D) 
\to (\overline{S}, F)$, for all $f\in \cE:= \cE(F(g, \G))$, where $D$ is a 
hypersurface arrangement in $\overline{M}$ with complement $F(g, \G)$, which induces $\cdga$
maps between Orlik--Solomon models, $f^*: A^{\hdot}(\overline{S}, F) \to  A^{\hdot}(\overline{M}, D)$, 
with the property that 
\begin{equation}
\label{eq:m3flat}
\F(A^{\hdot}(\overline{M}, D), \g) = \F^1(A^{\hdot}(\overline{M}, D), \g) \cup \bigcup_{f\in \cE} f^*\F(A^{\hdot}(\overline{S}, F), \g)\,
\quad \text{for $\g=\sl_2$ or $\sol_2$}\, , 
\end{equation}
and 
\begin{equation}
\label{eq:m3res}
\RR^1_1(A^{\hdot}(\overline{M}, D), \theta) = \Pi (A^{\hdot}(\overline{M}, D), \theta) \cup
\bigcup_{f\in \cE} f^*\F(A^{\hdot}(\overline{S}, F), \g)\, ,
\end{equation}
for any finite-dimensional representation $\theta : \g \to \gl (V)$.
\end{theorem}

This shows that, for partial configuration spaces on smooth projective curves, the 
non-trivial part at the origin, for both $\SL_2(\C)$--representation varieties and their 
degree one topological Green--Lazarsfeld loci, "comes from curves of general type, via
admissible maps". (The contribution of these curves, $f^*\F(A^{\hdot}(\overline{S}, F), \g)$,
was computed in \cite[Lemma 7.3]{MPPS}.) A similar pattern is exhibited by quasi-projective
manifolds with $1$--formal fundamental group, cf. \cite[Corollary 7.2]{MPPS}.
The geometric formulae from Theorem \ref{thm:main3} seem to be quite satisfactory, since in
genus $1$, where non-$1$-formal examples appear (cf. Theorem \ref{thm:main2}), the purely
algebraic description from \cite[Proposition 5.3]{MPPS} (obtained by assuming formality) may
not hold, as we explain in Example \ref{ex:notalg}.


\section{Admissible maps and rank one resonance}
\label{sec:pencils}

We devote this section to the proof of Theorem \ref{thm:main1}. Our strategy is to compute 
the irreducible decomposition of $\RR^1_1(A(g, \G))$, where $A^{\hdot}(g, \G)$ is the 
Orlik--Solomon model of $M:=F(g, \G)=\overline{M}\setminus D_{\G}$ from \cite{D},
$\overline{M}= \Sigma_g^{\sv}$ and $D_{\G}= \bigcup_{ij \in \se} \Delta_{ij}$. As a byproduct,
we obtain a complete description of the irreducible components through $1$, for the rank one
characteristic variety $\VV^1_1(P(g, \G))$, as explained in the Introduction.

The Dupont models $A^{\hdot}(\overline{M}, D)$ are defined over $\Q$ and generalize
Morgan's construction of {\em Gysin models} from \cite{M}, which corresponds to the case of a simple normal 
crossing divisor $D$. Among other things, the models of Dupont are natural with respect
to regular morphisms $\overline{f}: (\overline{M}, D) \to (\overline{M'}, D')$, in the
following sense. When the regular map $\overline{f}: \overline{M} \to \overline{M'}$ has the
property that $\overline{f}^{-1}(D') \subseteq D$, it induces a regular map
$f: \overline{M}\setminus D \to \overline{M'}\setminus D'$, and a $\cdga$ map 
$f^*: A^{\hdot}(\overline{M'}, D') \to A^{\hdot}(\overline{M}, D)$. Plainly, a graph inclusion
$f:\G' \inj \G$ (i.e., $f$ embeds $\sv'$ into $\sv$ and $\se'$ into $\se$) induces by
projection a regular morphism 
$\overline{f}: (\Sigma_g^{\sv}, D_{\G}) \to (\Sigma_g^{\sv'}, D_{\G'})$, and   a $\cdga$ 
map $f^*: A^{\hdot}(g, \G') \to A^{\hdot}(g, \G)$. Moreover, 
$A^{\hdot}(g, \G)= A^{\hdot}_{\hdot}(g, \G)$ is a bigraded $\cdga$ with {\em positive weights}, 
in the sense of Definition 5.1 from \cite{DP-ccm}. The lower degree, called {\em weight}, is
preserved by $\cdga$ maps induced by graph inclusions. A simple example:
$A^{\hdot}(g,\emptyset)=(H^{\hdot}(\Sigma_g^{\times n}),d=0)$.

Now, we recall from \cite{DP-ccm, MPPS} a couple of facts about rank $1$ resonance, needed 
in the sequel. Let $A^{\hdot}$ be a finite $\cdga$. For $\xi\in A^1\otimes \C=A^1$, the
Maurer--Cartan equation reduces to $d\xi=0$. Thus, $\F (A, \C)$ is naturally identified with
$H^1(A)\subseteq A^1$, since $A^0=\C \cdot 1$. By definition,
$\RR^1_1(A)= \{ \xi\in H^1(A) \mid H^1(A, d_{\xi})\ne 0\}$, where 
$d_{\xi}\eta= d\eta + \xi \eta$, for $\eta \in A^1$. Clearly, $\RR^1_1(A)$ 
depends only on the truncated $\cdga$ $A^{\le 2}:= A^{\hdot}/ \oplus_{i>2} A^i$, and 
$\RR^1_1(A)=\emptyset$ when $H^1(A)=0$. We will use the following consequence of Theorem C 
from \cite{DP-ccm}, applied to $M=F(g, \G)$ and $A=A(g, \G)$.

\begin{theorem}
\label{thm:respen}
For a quasi-projective manifold $M$ with finite model $A$ having positive weights, $\cE (M)$ 
is in bijection with the positive-dimensional (linear) irreducible components of 
$\RR^1_1(A)$, via the correspondence $f\in \cE (M) \mapsto \im H^1(f) \subseteq H^1(A)$.
\end{theorem} 

The maps from Theorem \ref{thm:main1} are constructed in the following way. For a subset 
$\sv'\subseteq \sv$, we denote by $\proj_{\sv'}: F(g, \G) \to F(g, \G')$ the regular map
induced by the canonical projection, $\proj_{\sv'}: \Sigma_g^{\sv} \to \Sigma_g^{\sv'}$,
where $\G'$ is the full subgraph of $\G$ with vertex set $\sv'$. For an elliptic curve 
$\Sigma_1$, let $\overline{\delta} : (\Sigma_1^2, \Delta_{12}) \to (\Sigma_1, \{ 0\})$ be 
the regular morphism defined by $\overline{\delta} (z_1, z_2)= z_1-z_2$. In genus $0$,
the regular map $\rho : F(0, K_4) \to \PP^1 \setminus \{ 0,1,\infty \}$ is defined by 
$\rho (z_1, z_2, z_3, z_4)=\alpha (z_4)$, where $\alpha \in \PSL_2$ is the 
unique automorphism of $\PP^1$ sending $z_1, z_2, z_3$ to $0,1,\infty $ respectively. For 
$g\ge 2$ and $f: K_1 \inj \G$, corresponding to 
$i\in \sv$, set $f_i:=\proj_i:F(g,\G)\to\Sigma_g$. For $g=1$ and $f:K_2\inj\G$, corresponding 
to $ij\in\se$, set $f_{ij}:=\delta \circ\proj_{ij}:F(1,\G)\to\Sigma_1\setminus\{0\}$. 
For $g=0$ and $f: K_4 \inj \G$, with vertex subset $\{ ijkl \} \subseteq \sv$, set 
$f_{ijkl}:= \rho \circ \proj_{ijkl} : F(0, \G) \to \PP^1 \setminus \{ 0,1,\infty \}$. 

\begin{lemma}
\label{lem:somepen}
The above maps, $f_i, f_{ij}$ and $f_{ijkl}$, are admissible, of general type.
\end{lemma}

\begin{proof}
In coordinates, $ \rho(z_1,z_2,z_3,z_4)=\frac{z_4-z_1}{z_2-z_1}:\frac{z_4-z_3}{z_2-z_3},\, 
\rho(0,1,\infty,z)=z. $ 
Obviously, the maps $\rho : F(0, K_4) \to \PP^1 \setminus \{ 0,1,\infty \}$,
$\delta : F(1, K_2) \to \Sigma_1\setminus\{0\}$ and the projections 
$\proj_* : F(g, \G) \to F(g, K_{\vert * \vert})$ (where $*$ stands for $i$ 
or $ij$ or $ijkl$ and $\vert * \vert$ is 1,2 or 4) are regular and surjective.
The general type condition is also clear: the spaces 
$\PP^1 \setminus \{ 0,1,\infty \}\simeq S^1\vee S^1\simeq \Sigma_1\setminus \{ 0\} $ have 
Euler characteristic $-1$ and $\chi(\Sigma_g)\leq -2$ for $g\geq 2$.

In order to finish the proof, we show that all the fibers are connected. Let us denote by $f_*$ any of the maps 
$f_i, f_{ij},f_{ijkl}$ and by $\varphi_*$ the restriction of $f_*$ to
$F(g,K_n)\subseteq F(g,\Gamma)$. The fiber $\varphi_*^{-1}(z)$ is dense in $f_*^{-1}(z)$
(fix one or two or four points and move the other points outside the diagonals $z_p=z_q$),
so it is enough to show that the fibers of $\varphi_*$ are connected. The fibers of 
$\delta$ and $\rho$ are path-connected:
$$ \Sigma_1\approx \delta^{\,-1}(z)\subseteq F(1,K_2), \, 
F(0,K_3)\approx \rho^{-1}(z)\subseteq F(0,K_4). $$
The fibers of $\varphi_*$ are path-connected as preimages of path-connected spaces through the 
locally trivial fibrations $\proj_*:F(g,K_n)\rightarrow F(g, K_{\vert *\vert})$ 
($\vert *\vert=1,2$ or 4) with path-connected fibers $F(\Sigma _g\setminus\{z_*\},K_{n-\vert *\vert})$.
\end{proof}

We recall from \cite[\S 6]{D} the complete description of the $\cdga$ $A^{\le 2}$, 
for $A:=A(g, \G)$.
We set $H^{\hdot}:= H^{\hdot}(\Sigma_g)$, with $H^2= \C \cdot \omega$ and with canonical 
symplectic basis
(for $g\geq 1$) of $H^1$, $\{ x^1,y^1, \dots, x^g,y^g \}$, with $x^sy^s=\omega$, for all $s$. 
We know from \cite{D} that
$A^{\hdot}$ is generated as an algebra  by $(H^{\hdot})^{\otimes \sv}$ (with weight equal to 
degree) and $G:= \spn \{ G_{ij} \mid ij\in \se \}$ (with degree $1$ and weight $2$). The 
bigraded $\cdga$ map $f^*: A^{\hdot}(g, \G') \to A^{\hdot}(g, \G)$, 
associated to $f:\G' \inj \G$, is determined by the canonical inclusions, 
$(H^{\hdot})^{\otimes \sv'} \inj (H^{\hdot})^{\otimes \sv}$ and $G' \inj G$. For $i\in \sv$ 
and $g\ge 0$, we set $f_i^* \omega := \omega_i$ and, for $g\geq 1$,  $f_i^*x^s:= x_i^s$ and 
$f_i^*y^s:= y_i^s$, for all $s$. The structure of the truncated algebra 
$A^{\le 2}=A^{\le 2}(g,\Gamma)$ is described by

\begin{itemize}
\item $A^1_1= H^1(\Sigma_g^{\sv})= \bigoplus_{i\in \sv} f_i^* H^1$; $A^1_2=G$

\item $A^2_2=  H^2(\Sigma_g^{\sv})$

\item $A^2_3= A^1_1 \otimes G$ modulo the relations (in genus $g\geq 1$) 
$(x_i^s-x_j^s)\otimes G_{ij}$ and $(y_i^s-y_j^s)\otimes G_{ij}$, for $s=1,\ldots,g$, 
$ij \in \se$ 

\item $A^2_4= \bigwedge^2 G$  modulo the relations $G_{jk}\wedge G_{ik}-G_{ij}\wedge G_{ik}+ 
G_{ij}\wedge G_{jk}$, for $f:K_3 \inj \G$; 
note that $A^2_4=\OS^2(\A_{\G})$, the degree $2$ piece of the Orlik--Solomon algebra \cite{OT}
of the associated graphic arrangement of hyperplanes in $\C^{\sv}$

\item  $d(A^1_1)=0$; $d(G_{ij})=\omega_i+\omega_j+\sum_s(y_i^s\otimes x_j^s-x_i^s\otimes y_j^s) 
\in A^2_2$ when $g\ge 1$ and $d(G_{ij})=\omega_i+\omega_j$ when $g=0$

\item $\mu : \bigwedge^2 G \to A^2_4$ is the quotient map (exactly as in the graded algebra $
\OS^{\hdot} (\A_{\G})$)

\item $\mu : \bigwedge^2 A^1_1 \to A^2_2$ is the cup-product in the cohomology ring $H^{\hdot}
(\Sigma_g^{\sv})$

\item $\mu : A^1_1 \otimes G \to A^2_3$ is the quotient map
\end{itemize}
(the lower indices of $f,x,y,\omega, G$ show the position in the cartesian or tensor product; the same 
convention will be used in Section 3 for $a,b,z,C$). 

\begin{lemma}
\label{lem:prel}
In degree one, we have the following:
\begin{enumerate}
\item \label{pr0}
If $g=0$, then $H^1(F(0, \G))= 0$ if and only if every connected component of 
$\Gamma$ is a tree or contains a unique cycle and this cycle has an odd length.
\item \label{pr1}
If $g\geq 1$, then $H^1(F(g, \G))=H^1(\Sigma_g^{\sv})\ne 0$.
\end{enumerate}
\end{lemma}

\begin{proof}
Due to the fact that $A$ is a model of $F(g, \G)$, we have
$$ H^1(F(g,\Gamma))=A_1^1\oplus \ker(d:A_2^1\to A_2^2)=H^1(\Sigma_g^{\sv})
\oplus \ker(d:G\to H^2(\Sigma_g^{\sv})).  $$

We can split the differential according to the connected components of the graph 
$\Gamma=\amalg\Gamma(\alpha)$, $\sv=\amalg\sv (\alpha)$, $G=\amalg G(\alpha)$:
$$ \ker(d:G\to H^2(\Sigma_g^{\sv}))=\oplus_{\alpha}\ker(d:G(\alpha)\to 
H^2(\Sigma_g^{\sv (\alpha)})), $$
so we give the proof for a connected graph $\Gamma$. 

For $g\geq 1$, the coefficient of $y_i^s\otimes x_j^s$ in the differential of 
$\gamma=\sum_{ij\in\se}t_{ij}G_{ij}$ is $t_{ij}$, therefore 
$d:G\to H^2(\Sigma_g^{\sv})$ is injective.

For $g=0$, $\gamma=\sum_{ij\in\se}t_{ij}G_{ij}$ is a cocycle if and only if
the coefficient of $\omega_i$ in $d(\gamma)$ is zero, i.e.
\begin{equation}
\label{eq:star} \sum_{j,\, ij\in\se}t_{ij}=0, \mbox{ for any }i\in\sv\, .
\end{equation}
This system of equations has $n$ equations and $|\se|$ unknowns; if 
$\chi(\Gamma)=n-|\se|<0$, one can find a non-trivial solution, hence $b_1(F(0, \G))\geq 1$. 
If $\chi(\Gamma)\geq 0$, we have to analyse only two cases (since $\Gamma$ is connected):

Case {\em a}: $\chi(\Gamma)=1$. In this case $\Gamma$ is a (finite) tree, hence it has a vertex $i$
of degree 1; one of the equations in the system (\ref{eq:star}) is $t_{ij}=0$ and  induction on $\vert \sv \vert$
applied to the tree $\Gamma\setminus\{i\}$ shows that the system has only the trivial 
solution (the induction starts with $n=1$, when $G=0$).

Case {\em b}: $\chi(\Gamma)=0$. In this case $\Gamma\simeq S^1$ contains a unique cycle $\Gamma_0$ and, 
possibly, some branches; starting with a vertex of degree 1, we can eliminate these branches
(if any), like in the previous case. The system is reduced to the equations corresponding to 
the vertices of $\Gamma_0$, say $1,2,\ldots,l$: 
$$ t_{i-1,i}+t_{i,i+1}=0, \, i \equiv 1,\ldots,l\, (\mbox{mod}\, l) \, .$$
We get a non-zero solution $(a,-a,a,\ldots,-a)$ only for $l$ even. 
\end{proof}

\begin{example} \mbox{}  
\begin{center}
\begin{picture}(360,60)
\put(20,20){$\Gamma_1:$}                 \multiput(80,50)(40,-30){2}{$\bullet$}
\multiput(70,10)(20,0){2}{$\bullet$}     \multiput(70,30)(20,0){2}{$\bullet$}
\multiput(80,20)(0,20){2}{$\bullet$}     \multiput(100,10)(0,10){2}{$\bullet$}
\multiput(140,20)(0,20){2}{$\bullet$}    \multiput(150,30)(40,0){2}{$\bullet$}
\multiput(170,40)(0,10){2}{$\bullet$}    \multiput(160,10)(20,0){2}{$\bullet$}
\multiput(210,20)(0,20){2}{$\bullet$}    \multiput(230,10)(0,20){2}{$\bullet$}
\put(200,30){$\bullet$}    \put(93,13){\line(1,0){10}}     \put(93,13){\line(1,1){10}}
\put(295,20){$b_1(F(0,\Gamma_1))=0$}     \multiput(73,13)(0,20){2}{\line(1,1){10}}
\multiput(83,23)(0,20){2}{\line(1,-1){10}}  \put(83,23){\line(0,1){30}}
\put(143,23){\line(1,1){10}}             \put(143,43){\line(1,-1){10}}
\put(153,33){\line(2,1){20}}             \put(153,33){\line(1,-2){10}}
\put(163,13){\line(1,0){20}}             \put(173,43){\line(0,1){10}}
\put(173,43){\line(2,-1){20}}            \put(183,13){\line(1,2){10}}
\put(193,33){\line(1,0){10}}             \put(233,13){\line(0,1){20}}
\put(203,33){\line(1,1){10}}             \put(203,33){\line(1,-1){10}}
\end{picture}
\end{center}
\end{example}

\begin{example} Every edge is marked with its coefficient in an arbitrary cocycle; 
the unmarked edges have coefficient 0.  
\begin{center}
\begin{picture}(360,65)
\put(20,30){$\Gamma_2:$}                      \multiput(70,20)(40,0){2}{$\bullet$}
\multiput(90,30)(0,20){2}{$\bullet$}          \multiput(140,30)(20,0){2}{$\bullet$}
\multiput(190,30)(0,20){2}{$\bullet$}         \multiput(210,40)(35,0){2}{$\bullet$}
\multiput(265,30)(0,20){2}{$\bullet$}         \put(150,50){$\bullet$}
\put(75,40){$a$}                              \put(98,30){$a$}
\put(103,40){$b$}                             \put(83,24){$b$}
\put(89,15){$c$}        \put(87,40){$c$}      \put(100,7){$(a+b+c=0)$}
\put(295,30){$b_1(F(0,\Gamma_2))=3$}          \multiput(175,39)(95,0){2}{$-d$}
\multiput(200,29)(55,0){2}{$d$}               \multiput(200,50)(55,0){2}{$d$}
\put(219,45){$-2d$}                           \put(73,23){\line(1,0){40}}
\put(73,23){\line(2,3){20}}                   \put(73,23){\line(2,1){20}}
\put(93,33){\line(0,1){20}}                   \put(93,33){\line(2,-1){20}}
\put(93,53){\line(2,-3){20}}                  \put(143,33){\line(1,0){20}}
\multiput(193,33)(55,10){2}{\line(2,1){20}}   \put(143,33){\line(1,2){10}}             
\multiput(193,53)(55,-10){2}{\line(2,-1){20}} \put(153,53){\line(1,-2){10}}
\multiput(193,33)(75,0){2}{\line(0,1){20}}    \put(213,43){\line(1,0){35}}
\end{picture}
\end{center}
\end{example}

\begin{remark}
More generally, let $\Sigma$ be an arbitrary complex projective manifold of dimension $m\ge 1$.
The full configuration space $F(\Sigma, K_n)$ has a remarkable $\cdga$ model, $E^{\hdot}(\Sigma, n)$;
when $m=1$, $E^{\hdot}(\Sigma_g, n)= A^{\hdot}(g, K_n)$ (see e.g. \cite{AA} for details and references
related to these models). As a graded algebra, $E^{\hdot}(\Sigma, n)$ is generated by $H^{\hdot}(\Sigma^n)$
and $G:= \spn \{ G_{ij} \mid 1\le i<j\le n \}$, taken in degree $2m-1$. Denote by $EE^{\hdot}(\Sigma, n)$
the graded subalgebra of $E^{\hdot}(\Sigma, n)$ generated by $G$. It is shown in \cite{AA} that, when
$\Sigma \ne \Sigma_0$, the restriction of $d$ to $EE^{+}(\Sigma, n)$ is injective. This more general result 
gives an alternative proof of Lemma \ref{lem:prel}\eqref{pr1}. 
\end{remark}

\begin{prop}
\label{prop:res2}
When $g\ge 2$, $\RR^1_1(A(g, \G))= \bigcup_{i\, \in \sv} \im H^1(f_i)$ is the irreducible 
decomposition.
\end{prop}

\begin{proof}
The inclusion $ \bigcup_{i\, \in \sv} \im H^1(f_i)\subseteq \RR^1_1(A(g, \G))$ is an obvious 
consequence of Theorem \ref{thm:respen} and Lemma \ref{lem:somepen}. 
For the proof of the opposite inclusion, we start with a 
non-zero cohomology class $\xi$ in $H^1(A)$ and a $d_{\xi}$-cocycle $\eta \not\in \C \cdot \xi$:
$$ \xi=\Sigma_{i,s}(p_i^sx_i^s+q_i^sy_i^s),\, \eta=
    \Sigma_{i,s}(u_i^sx_i^s+v_i^sy_i^s)+\Sigma_{ij\in\se}t_{ij}G_{ij} $$
(from Lemma \ref{lem:prel}\eqref{pr1}, $\xi$ has no component in $G$). For an arbitrary $\eta$, the 
differential $d_{\xi}\eta=d\eta+\xi\cdot\eta$ belongs to $A_2^2\oplus A_3^2$; these two 
components are:
$$\begin{array}{lll}
A_2^2 & \ni & \Sigma_{ij\in\se}t_{ij}(\omega_i+\omega_j+\Sigma_s(y_i^s\otimes x_j^s-x_i^s
\otimes y_j^s))+\Sigma_{i,s}(p_i^sx_i^s+q_i^sy_i^s)\cdot \Sigma_{i,s}(u_i^sx_i^s+v_i^sy_i^s) \\
A_3^2 & \ni & \Sigma_{i,s}(p_i^sx_i^s+q_i^sy_i^s)\cdot (\Sigma_{ij\in\se}t_{ij}G_{ij})=
\xi\cdot \gamma\, .
\end{array}$$ 

We will show that the $G$-component of the $d_{\xi}$-cocycle $\eta$, 
$\gamma=\Sigma_{ij\in\se}t_{ij}G_{ij}$, is 0. Otherwise, there is an edge $ij$ with 
$t_{ij}\ne 0$. Since the annihilator of $G_{hk}$ is the span of 
$\{x_h^s-x_k^s,y_h^s-y_k^s\}_{1\le s\le g}$, the vanishing of the $A_3^2$-component of 
$d_{\xi}\eta$ implies that $\xi$ is reduced to 
$$ \xi=\Sigma_sp^s(x_i^s-x_j^s)+\Sigma_s q^s(y_i^s-y_j^s) $$
and also that $\gamma$ has only one non-zero coefficient
$t_*$ (we can normalize it: $t_{ij}=1$). In $A_2^2$, if $h\ne i,j$, the coefficients of 
$x_i^s\otimes x_h^r$, $x_i^s\otimes y_h^r$, $y_i^s\otimes x_h^r$ and $y_i^s\otimes y_h^r$ 
should be 0, hence $u_h^s=v_h^s=0$ for any $h\ne i,j$ and any $s$. Hence, the 
$A_2^2$-component of $d_{\xi}\eta$ is reduced to
$$ \omega_i+\omega_j+\Sigma_s(y_i^s\otimes x_j^s-x_i^s\otimes y_j^s)+(\Sigma_sp^s(x_i^s-x_j
^s)+\Sigma_s q^s(y_i^s-y_j^s))\cdot \Sigma_s(u_i^sx_i^s+u_j^sx_j^s+v_i^sy_i^s+v_j^sy_j^s);$$
the coefficients of the following elements in the canonical basis of $A_2^2$ are 0:
$$ \begin{array}{cccc}
   \omega_i  &  x_i^s\otimes y_j^s  &  y_i^s\otimes x_j^s  &  x_i^r\otimes x_j^s \\ 
1+\Sigma_sp^sv_i^s-\Sigma_sq^su_i^s & -1+p^sv_j^s+q^su_i^s & 1+q^su_j^s+p^sv_i^s &
 p^ru_j^s+p^su_i^r \,.
 \end{array}  $$
We show that this system has no solution. By the symmetry $(p,x)\leftrightarrow(q,y)$, we can 
suppose that there is an index $s$ such that $p^s\ne 0$; if some $p^r=0$, the second equation 
(for $s\to r$) implies that $u_i^r\ne 0$ and from the last equation we get $p^s=0$, a 
contradiction. If all
the coefficients $p^s$ are non-zero, the last equation (for $s=r$) implies that, for any $s$,
$u_j^s=-u_i^s$ and the third equation shows that $1-q^su_i^s+p^sv_i^s=0$, for any $s$.
Adding these $g$ equations we find $g+\Sigma_sp^sv_i^s-\Sigma_sq^su_i^s=0$, and, comparing with 
the first equation, we obtain $g=1$, again a contradiction. 

Therefore, $\gamma=0$; the non-vanishing of $H^1(A,d_{\xi})$ is equivalent to:
$$ d_{\xi}\eta=\xi\cdot \eta=0,\,  \eta \not\in \C \cdot \xi \, . $$
This implies that $\xi \in \RR^1_1(H^{\hdot}(\Sigma_g)^{\otimes \sv}, d=0)$.
We infer from the K\"unneth formula for resonance \cite[Proposition 5.6]{M10} that
$\xi\in \im H^1(f_i)$, for some $i\in \sv$. 

In conclusion, $\RR^1_1(A)= \bigcup_{i\, \in \sv} \im H^1(f_i)$ is a finite union of linear subspaces.
Since clearly there are no redundancies, this is the irreducible decomposition, as claimed.
\end{proof}

\begin{prop}
\label{prop:res1}
When $g=1$, $\RR^1_1(A(1, \G))= \bigcup_{ij\, \in \se} \im H^1(f_{ij})$ is the irreducible 
decomposition, if $\se \ne \emptyset$. Otherwise, $\RR^1_1(A(1, \G))= \{ 0\}$. 
\end{prop}

\begin{proof}
Suppose that $\se=\emptyset$. As mentioned before, $A(1,\emptyset)=(\bigwedge(x_i,y_i),d=0)$, and it is 
well known that the resonance variety $\RR^1_1$ of an exterior algebra is reduced to 0. 

Suppose that $\se$ is non-empty. Given a non-zero cohomology class 
$\xi=\Sigma_ip_ix_i+\Sigma_iq_iy_i \in \RR^1_1(A)$ 
(see Lemma \ref{lem:prel}\eqref{pr1}), we may find
$\eta=\Sigma_iu_ix_i+\Sigma_iv_iy_i+\Sigma_{ij\in\se}t_{ij}G_{ij}$ such that 
$d_{\xi}\eta=0$ and $\eta \not\in \C \cdot \xi $. We may also suppose that there is one 
coefficient $t_{ij}\ne 0$ (otherwise we are in the previous case). Now we can apply the
argument given in the proof of Proposition \ref{prop:res2}: there is only one non-zero
coefficient $t_*$ and $\xi \in \mathrm{Ann}(G_{ij})$, hence $\xi=p(x_i-x_j)+q(y_i-y_j)$. 
On the other hand, it is obvious that $H^1(f_{ij})(z)=z_i-z_j$, for 
$z\in H^1(\Sigma_1 \setminus \{ 0\})=H^1(\Sigma_1)$. 

We conclude, like in the proof of Proposition \ref{prop:res2}, that 
$\RR^1_1(A)= \bigcup_{ij\, \in \se} \im H^1(f_{ij})$ is the irreducible decomposition, 
in this case.
\end{proof}

\begin{prop}
\label{prop:res0}
When $g=0$ and $H^1(A(0, \G))=0$, $\RR^1_1(A(0, \G))= \emptyset$. 

When $H^1(A(0,\G))\ne 0$, 
$\RR^1_1(A(0, \G))=\{ 0\} \cup \bigcup \im H^1(f_{ijkl})$ is the irreducible decomposition,
where the union is taken over all $K_4$--subgraphs of $\G$ with vertex set $\{ ijkl \}$, and $\{ 0\}$ 
is omitted when $\G$ contains such a subgraph. 
\end{prop}

\begin{proof}
If $H^1(A(0, \G))=0$ and $\xi\in \RR^1_1(A)$, the definitions imply that $d_{0}\eta=d\eta=0$,
for some $\eta \in A^1$. From this we get $\eta=0$, which shows that $\RR^1_1(A(0, \G))= \emptyset$.

From now on, we assume $H^1(A) \ne 0$. For any $K_4\inj \G$ on vertices $i,j,k,l$, let us denote by 
$R_{ijkl} \subseteq H^1(A)$ the $2$--dimensional subspace 
$\{ a(G_{ij}+G_{kl})+ b(G_{ik}+G_{jl})+c(G_{jk}+G_{il}) \, \mid \, a+b+c=0\}$. When $\G=K_4$,
we find that $H^1(A(0, K_4))=R_{1234}$, by solving the system \eqref{eq:star}. 
The map $H^1(f_{ijkl})$ is injective, since $f_{ijkl}$ is admissible. Therefore, 
$\im H^1(f_{ijkl})=R_{ijkl}$.

The inclusion $\RR^1_1(A) \supseteq \{ 0\} \cup \bigcup R_{ijkl}$ follows from 
Theorem \ref{thm:respen} and Lemma \ref{lem:somepen}. Since plainly there are no redundancies 
in the above finite union of linear subspaces, we are left with proving that 
$\RR^1_1(A) \setminus \{ 0\} \subseteq \bigcup R_{ijkl}$. To achieve this, we will also need to consider,
for any $K_3\inj \G$ on vertices $i,j,k$, the linear subspace $R_{ijk} \subseteq G= \OS^1(\A_{\G})$
defined by $R_{ijk}= \{ a G_{ij}+ bG_{jk}+ cG_{ik} \mid a+b+c=0 \}$.

If $\xi \in \RR^1_1(A) \setminus \{ 0\} \subseteq G \setminus \{ 0\}$, then $d\xi=0$ and there is 
$\eta\in G\setminus \C \cdot \xi$ such that $d_{\xi}\eta= d\eta+ \xi \cdot \eta=0\in A_2^2 \oplus A_4^2$,
or, equivalently, $d\eta=0$ and $ \xi \cdot \eta=0\in \OS^2(\A_{\G})$. In particular,
$\xi \in \RR^1_1( \OS^{\hdot}(\A_{\G}), d=0) \setminus \{ 0\}$. It follows from \cite[\S 3.5]{SS}
that either $\xi \in R_{ijk}$ for some $K_3\inj \G$, or  $\xi \in R_{ijkl}$ for some $K_4\inj \G$.

The first case cannot occur, since clearly $R_{ijk} \cap \ker(d)= 0$, by \eqref{eq:star}, and we are done.
\end{proof}

Theorem \ref{thm:respen} and Lemma \ref{lem:somepen}, together with Propositions \ref{prop:res2}--\ref{prop:res0},
prove Theorem \ref{thm:main1} from the Introduction. In the genus 0 case, the implication
"$H^1(A(0, \G))=0 \Longrightarrow \G$ has no $K_4$--subgraphs" is provided by Lemma \ref{lem:prel}\eqref{pr0}.


\section{Malcev completion and formality}
\label{sec:mal}

We continue our analysis of partial pure braid groups with the proof of Theorem 
\ref{thm:main2}. Their Malcev Lie algebras are computed with 
the aid of the holonomy Lie algebras of their Orlik--Solomon models, $A^{\hdot}(g, \G)$. 

We will also consider a weaker notion of $1$--formality: a finitely generated group $\pi$ is 
{\em filtered formal} if its Malcev Lie algebra $\m (\pi)$ is isomorphic to the $\lcs$ 
completion of a Lie algebra presentable with degree $1$ generators and relations 
homogeneous with respect to bracket length. We recall that the free Lie algebra on a vector 
space, $\L^{\hdot} (W)$, is graded by bracket length. In low degrees, $\L^1 (W)=W$, and the 
Lie bracket identifies $\L^2 (W)$ with $\bigwedge^2 W$.

We are going to make extensive use of the following construction, introduced in 
\cite[Definition 4.2]{MPPS}. The holonomy Lie algebra of a $1$--finite 
$\cdga$ $A$, $\h (A)$, is the quotient of $\L (A^{1*})$ by the Lie ideal generated by
$\im (d^*+\mu^*)$, where $d:A^1 \to A^2$ (respectively $\mu: \bigwedge^2 A^1 \to A^2$) is the 
differential (respectively the product) of the $\cdga$ $A^{\le 2}$, and $(\cdot)^*$ denotes 
vector space duals. This Lie algebra is functorial with respect to $\cdga$ maps, and has the 
following basic property. (A result similar to our theorem below was proved by Bezrukavnikov in
\cite{B}, under the additional assumption that $A^{\hdot}$ is quadratic as a graded algebra; note
that this condition is not satisfied in general by finite $\cdga$ models of spaces, in particular by
the  models $A^{\hdot}(0, \G)$.)

\begin{theorem}
\label{thm:malhol}
If $A$ is a $1$--finite $1$--model of a connected CW-space $M$, 
then $\m (\pi_1(M))$ is isomorphic to the $\lcs$ completion of $\h (A)$, as filtered Lie algebras. 
\end{theorem}

\begin{proof}
Our approach is based on a key result obtained by Chen in \cite{Ch} and refined by Hain in \cite{Ha1}.
This result provides the following description for the Malcev completion of $\pi:=\pi_1(M)$, over a
characteristic zero field $\k$, in terms of iterated integrals and bar constructions.

Consider the complete Hopf algebra $\widehat{\k \pi}$, where the completion is taken with respect to 
the powers of the augmentation ideal of the group ring $\k \pi$. The complete Lie algebra $\m (\pi)$
is the Lie algebra of primitives, $P\widehat{\k \pi}$, endowed with the induced filtration, defined by
Quillen in \cite[Appendix A]{Q}. On the other hand, let $B^{\hdot}(A)$ be the differential graded Hopf algebra
obtained by applying the bar functor to the augmented $\cdga$ $A^{\hdot}$, where the augmentation sends
$A^+$ to $0$ and is the identity on $A^0=\k \cdot 1$; see e.g. \cite[\S 1.1]{Ha1}. The dual Hopf algebra,
$H^0 B(A)^*= \Hom_{\k} (H^0 B(A), \k)$, is a complete Hopf algebra, with filtration induced from the 
bar filtration of $H^0 B(A)$; see \cite[\S 2.4]{Ha1}. 

Next, let $f:A' \to A''$ be an augmented $\cdga$ map inducing an isomorphism in $H^i$ for $i\le 1$ and a
monomorphism in $H^2$ (for short, $f$ is an augmented $1$--equivalence). If $H^0 (A')=\k \cdot 1$,
we claim that the induced map, $H^0 B(f)^* \colon H^0 B(A'')^* \to H^0 B(A')^*$, is a filtered isomorphism.
Indeed, a standard argument based on the Eilenberg--Moore spectral sequence (like in Proposition 1.1.1 from \cite{Ha1})
shows that $H^0 B(f)$ induces an isomorphism at the associated graded level, with respect to the bar filtrations,
which clearly implies our assertion. The fact that $A^{\hdot}$ and $\Omega^{\hdot} (M)$ have the same Sullivan
$1$--minimal model, $\mathcal{M}^{\hdot}$, implies by rational homotopy theory \cite{S} the existence of two
augmented $1$--equivalences, $\mathcal{M}^{\hdot} \to A^{\hdot}$ and $\mathcal{M}^{\hdot} \to \Omega^{\hdot} (M)$.
Here, both $A^{\hdot}$ and $\mathcal{M}^{\hdot}$ are canonically augmented, as above, since 
$A^0=\mathcal{M}^0 =\k \cdot 1$, and the augmentation of $\Omega^{\hdot} (M)$ is induced by the basepoint 
chosen for $\pi_1 (M)$, as in \cite{Ha1}.

It follows from \cite[Corollary 2.4.5]{Ha1} that integration induces an isomorphism between $\widehat{\k \pi}$
and $H^0 B(A)^*$, as complete Hopf algebras. This leads to the aforementioned description of the Malcev Lie algebra:
$\m (\pi) \simeq PH^0 B(A)^*$, as complete Lie algebras.

Now, we claim that we may assume that $A^{\hdot}$ is of finite type, i.e., all graded pieces are finite-dimensional.
Indeed, the canonical $\cdga$ projection, $A^{\hdot} \surj A^{\le 2}$, is clearly a $1$--equivalence. Hence, $A^{\le 2}$
is also a $1$--model of $M$, by \cite{S}. It is equally easy to check that 
$\iota : \k \cdot 1 \oplus A^1 \oplus (\im (d)+\im (\mu)) \inj A^{\le 2}$ is a $\cdga$ inclusion and a $1$--equivalence. 
Therefore, we may replace $A^{\le 2}$ by the above finite type sub--$\cdga$, without changing the holonomy Lie algebra,
as claimed.

We may thus consider the dual cocommutative differential graded coalgebra, $A_{\hdot}:=A^{\hdot \, *}$. By the 
standard duality between the bar construction for $\cdga$'s and the Adams cobar construction $C$ for 
cocommutative differential graded coalgebras \cite{Ad}, the complete Hopf algebras $H^0 B(A^{\hdot})^*$ and 
$\widehat{H_0} C(A_{\hdot})$ are isomorphic. In concrete terms, the Hopf algebra $H_0 C(A_{\hdot})$ is easily
identified with the quotient of the primitively generated tensorial Hopf algebra on $A_1$, by the two-sided Hopf ideal
generated by $\im (-d^* +\mu^*)$, and the completion is taken with respect to the descending filtration 
induced by tensor length. 

Denote by $\mathfrak{q}(A)$ the quotient of the free Lie algebra $\L (A_{1})$ by the Lie ideal generated by 
$\im (-d^* +\mu^*)$. The above discussion shows that the complete Hopf algebras $H^0 B(A)^*$ and
$\widehat{U} \mathfrak{q}(A)$ are isomorphic, where $\widehat{U}$ is Quillen's completed universal enveloping 
algebra functor from \cite[Appendix A]{Q}. 

Plainly, $-\id \colon A_1 \to A_1$ induces an isomorphism between the Lie algebras $\mathfrak{q}(A)$ and
$\mathfrak{h}(A)$. We infer that $\m (\pi) \simeq P\widehat{U} \mathfrak{h}(A)$, as complete Lie algebras.

Finally, let $\h$ be a Lie algebra, and consider the canonical Lie homomorphism from \cite[Appendix A]{Q},
$\kappa \colon \h \to P\widehat{U} \h$. By \cite[A3.9 and A3.11]{Q}, $\kappa$ sends the $\lcs$ filtration of $\h$
into the Malcev filtration of $P\widehat{U} \h$, inducing an isomorphism at the associated graded level. Passing to
completions, we infer that $\widehat{\kappa} : \widehat{\h} \to P\widehat{U} \h$ is a filtered Lie isomorphism.
We conclude that $\m (\pi) \simeq \widehat{\h (A)}$, as filtered Lie algebras, thus finishing our proof.
\end{proof}

When $M=F(g, \G)$ and $A=A(g, \G)$, set $L(g, \G):=\h (A(g, \G))$. We will denote, for $g\ge 0$, the basis dual to 
$\{G_{ij} \}_{ij\, \in \se}$ and $\{\omega_i\}_{i\in\sv}$ by $\{C_{ij}\}_{ij\, \in \se}$ and 
$\{z_i\}_{i\in\sv}$ respectively. For $g\geq 1$, the basis dual to 
$\{ x_i^s, y_i^s \mid 1\le i\le n, 1\le s\le g \}$ will be denoted $\{ a^s_i, b^s_i\}$.

\begin{prop}
\label{prop:mal0}
The Malcev Lie algebra $\m (P(0, \G))$ is isomorphic to the $\lcs$ completion of $L(0, \G)$, 
where the Lie algebra $L(0, \G)$ is the quotient of the free Lie algebra on 
$\{ C_{ij}\}_{ij\, \in \se}$ by the relations
\begin{equation}
\label{eq:m0l}
\sum_{j,\,ij\in\se} C_{ij} \quad (i\in \sv)
\end{equation}
\begin{equation}
\label{eq:m0q1}
[C_{ij},C_{kl}] \quad (ij, kl\, \in \se)
\end{equation}
\begin{equation}
\label{eq:m0q2}
[C_{ij},C_{jk}] \quad (ij, jk\, \in \se\mbox{ \em and }\mbox{ik}\not\in\se)
\end{equation}
\begin{equation}
\label{eq:m0q3}
[C_{ij}+C_{jk},C_{ik}] \quad (ij, jk, ik\, \in \se )
\end{equation}
In particular, the group $P(0, \G)$ is always $1$--formal.
\end{prop}

\begin{proof}
We consider the following canonical basis in $(A^2)^*$:
$$ \{z_i\}_{i\in\sv}\cup \{C_{ij}\wedge C_{kl}\}_{ij,kl\in\se}
\cup \{C_{ij}\wedge C_{jk}\}_{ik\notin\se}\cup
\{C_{ij}\wedge C_{ik},C_{ij}\wedge C_{jk}\}_{ij,ik,jk\in\se} $$ 
(in the product $C_{ij}\wedge C_{kl}$ we take $j>i<k<l, j\ne k,l$ and in the last set
we take $i<j<k$, see \cite{B}). Dualizing $d$ and $\mu$, where
$$ dG_{ij}=\omega_i+\omega_j\, , \, \mu(G_{ik}\wedge G_{jk})=
  G_{ij}\wedge G_{jk}-G_{ij}\wedge G_{ik} \, ,$$
we obtain the defining relations in the last row of the table (in the last two columns 
$i<j<k$):
\begin{center}
\begin{tabular}{|c|c|c|c|c|c|}
\hline
{} & $z_i$     & $C_{ij}\wedge C_{kl}$        & $C_{ij}\wedge C_{jk}$    & 
                 $C_{ij}\wedge C_{ik}$        & $C_{ij}\wedge C_{jk}$    \\ 
{} & $i\in\sv$ & $\mathrm{card}\{i,j,k,l\}=4$ & $(ik\notin\se)$          &
                 $ij,ik,jk\in\se$             & $ij,ik,jk\in\se$         \\
\hline 
$d^*$ & $\Sigma_{j,\,ij\in\se}C_{ij}$ & $0$ & $0$ & $0$ & $0$ \\
\hline
$\mu^*$ & $0$ & $[C_{ij},C_{kl}]$ & $[C_{ij},C_{jk}]$ & $[C_{ij}+C_{jk},C_{ik}]$ &
          $[C_{ij}+C_{ik},C_{jk}]$ \\
\hline
$\Downarrow$ & (\ref{eq:m0l}) & (\ref{eq:m0q1}) & (\ref{eq:m0q2}) & (\ref{eq:m0q3}) &     
               (\ref{eq:m0q3}) \\
\hline
\end{tabular} 
\end{center}
From the last two relations we obtain $[C_{ik}+C_{jk},C_{ij}]=0$, hence the relation 
(\ref{eq:m0q3}), where $i,j,k$ are arbitrarily ordered.  
\end{proof}
 
\begin{remark}
\label{rk:w1} 
By \cite[Corollary 10.3]{M}, if the quasi-projective manifold $M$ has the vanishing property 
in degree $1$ $W_1H^1(M)=0$, then $\pi_1(M)$ is $1$--formal,
where $W_{\hdot}$ denotes Deligne's weight filtration \cite{De2, De3}. According to \cite{De2, De3}, 
$W_1H^1(M)=0$ whenever $M$ admits a smooth compactification $\overline{M}$ with 
$b_1(\overline{M})=0$. Hence, $P(0, \G)$ is actually $1$--formal in this stronger sense. 
\end{remark}

\begin{prop}
\label{prop:malpos}
For $g\ge 1$, the Malcev Lie algebra $\m (P(g, \G))$ is isomorphic to the $\lcs$ completion of 
$L(g, \G)$, where the Lie algebra $L(g, \G)$ is the quotient of the free Lie algebra on 
$\{ a^s_i, b^s_i\}$ by the relations
\begin{equation}
\label{eq:malpos21}
C_{ij}:= [a^s_i, b^s_j]= [a^t_j, b^t_i] \quad (\forall i\ne j\, , \forall s,t)
\end{equation}
\begin{equation}
\label{eq:malpos22}
C_{ij}=0 \quad (ij\, \not\in \se)
\end{equation}
\begin{equation}
\label{eq:malpos23}
[a^s_i, b^t_j]= [a^s_j, b^t_i]=0 \quad (\forall i<j\, , \forall s\ne t)
\end{equation}
\begin{equation}
\label{eq:malpos24}
[a^s_i, a^t_j]= [b^s_i, b^t_j]=0 \quad (\forall i\ne j\, , \forall s,t)
\end{equation}
\begin{equation}
\label{eq:malpos25}
\sum_j C_{ij}= \sum_s [b^s_i, a^s_i] \quad (i\, \in \sv)
\end{equation}
\begin{equation}
\label{eq:malpos26}
[a^s_k,C_{ij}]= [b^s_k,C_{ij}]=0 \quad (\forall k\ne i,j\,, \forall s )
\end{equation}
In particular, $L(g, \G)$ is generated in degree $1$ with relations in degrees $2$ and $3$, and
consequently the group $P(g, \G)$ is always filtered formal. 
\end{prop}

\begin{proof}
The canonical basis in $(A^2)^*$ contains the list in the proof of Proposition \ref{prop:mal0}
and also (with indices $1\leq i<j\leq n, 1\leq s,t\leq g,k\neq i,j$)
$$ \{a_i^s\otimes a_j^t,a_i^s\otimes b_j^t,b_i^s\otimes a_j^t,b_i^s\otimes b_j^t\}\cup
\{a_k^s\otimes C_{ij},b_k^s\otimes C_{ij},a_i^s\otimes C_{ij},b_i^s\otimes C_{ij}\}. $$
To dualize $d$ and $\mu$, the relevant relations are
$$ \begin{array}{l}
dG_{ij}=\omega_i+\omega_j+\Sigma_s(y_i^s\otimes x_j^s-x_i^s\otimes y_j^s),\\
\mu(x_i^s\wedge y_i^s)=\omega_i, \,\mu(x_i^s\wedge y_j^{\,t})=x_i^s\otimes y_j^{\,t}, \,
  \mu(y_i^s\wedge x_j^{\,t})=y_i^s\otimes x_j^{\,t}\, (i<j),\\
\mu(x_i^s\wedge x_j^{\,t})=x_i^s\otimes x_j^{\,t}, \, 
    \mu(y_i^s\wedge y_j^{\,t})=y_i^s\otimes y_j^{\,t}\, (i<j),\\
\mu(x_i^s\wedge G_{jk})=x_i^s\otimes G_{jk}, \, \mu(y_i^s\wedge G_{jk})=y_i^s\otimes G_{jk},\\
\mu(x_i^s\wedge G_{ij})=x_i^s\otimes G_{ij}=\mu(x_j^s\wedge G_{ij}), \, 
\mu(y_i^s\wedge G_{ij})=y_i^s\otimes G_{ij}=\mu(y_j^s\wedge G_{ij}).  
\end{array}$$
The defining relations are obtained in the last row of the following table. The 
indices in columns 4 and 5 satisfy $i<j<k$ (for any $C_{pq}$ in the table, $pq\in\se$,
and the entries in columns 6 and 7 are to be replaced by $0$ in the second row, when $ij \not\in \se$):

\begin{center}
\begin{tabular}{|c|c|c|c|c|c|}
\hline
0 & 1 & 2 & 3 & 4 & 5 \\
\hline
{} & $z_i$     & $C_{ij}\wedge C_{kl}$        & $C_{ij}\wedge C_{jk}$    & 
                 $C_{ij}\wedge C_{ik}$        &   $C_{ij}\wedge C_{jk}$  \\  
{} & $i\in\sv$ & $\mathrm{card}\{i,j,k,l\}=4$ & $(ik\notin\se)$          &
                 $ij,ik,jk\in\se$             & $ij,ik,jk\in\se$         \\
\hline 
$d^*$ & $\Sigma_{j,\,ij\in\se}C_{ij}$ & $0$ & $0$ & $0$ & $0$ \\
\hline
$\mu^*$ & $\Sigma_s[a_i^s,b_i^s]$  & $[C_{ij},C_{kl}]$        & $[C_{ij},C_{jk}]$ & 
          $[C_{ij}+C_{jk},C_{ik}]$ & $[C_{ij}+C_{ik},C_{jk}]$ \\
\hline
$\Downarrow$ & (\ref{eq:malpos25}) & (\ref{eq:malpos31}) & (\ref{eq:malpos34}) & 
               (\ref{eq:malpos33}) & (\ref{eq:malpos33}) \\
\hline
\end{tabular} 
\end{center}
\begin{center}
\begin{tabular}{|c|c|c|c|c|c|c|c|}
\hline
6 & 7 & 8 & 9 & 10 & 11 & 12 & 13 \\
\hline
   $a_i^s\otimes b_j^t$ & $b_i^s\otimes a_j^t$ & $a_i^s\otimes a_j^t$ &
   $b_i^s\otimes b_j^t$ & $a_k^s\otimes C_{ij}$ & $b_k^s\otimes C_{ij}$ & 
   $a_i^s\otimes C_{ij}$ & $b_i^s\otimes C_{ij}$ \\  
   $i<j$                & $i<j$                & $i<j$                &
   $i<j$                & $k\ne i,j$           & $k\ne i,j$           &
   $i<j$                & $i<j$                \\
\hline 
   $-\delta_{st}C_{ij}$ & $\delta_{st}C_{ij}$ & $0$ & $0$ & $0$ & $0$ & $0$ & $0$ \\
\hline
   $[a_i^s,b_j^t]$  & $[b_i^s,a_j^t]$  & $[a_i^s,a_j^t]$  & $[b_i^s,b_j^t]$ & 
   $[a_k^s,C_{ij}]$ & $[b_k^s,C_{ij}]$ & $[a_i^s+a_j^s,C_{ij}]$ & 
   $[b_i^s+b_j^s,C_{ij}]$ \\ 
\hline
   (\ref{eq:malpos21}-\ref{eq:malpos23})     & (\ref{eq:malpos21}-\ref{eq:malpos23}) & 
   (\ref{eq:malpos24}) & (\ref{eq:malpos24}) & (\ref{eq:malpos26})  & 
   (\ref{eq:malpos26}) & (\ref{eq:malpos32}) & (\ref{eq:malpos32})  \\
\hline
\end{tabular} 
\end{center}

Note that, when $ij \in \se$, in the relations \eqref{eq:malpos21} $C_{ij}$ is the dual of $G_{ij}$.
The relations (\ref{eq:malpos26}) are obtained in columns 
10 and 11 for $ij\in\se$ and, otherwise, are a trivial consequence of (\ref{eq:malpos22}).  
It remains to prove that the relations (\ref{eq:malpos21})-(\ref{eq:malpos26}) imply the 
following list
\begin{equation}
\label{eq:malpos31}
[C_{ij},C_{kl}]=0 \quad (\mbox{if }\mathrm{card}\{i,j,k,l\}=4)
\end{equation}
\begin{equation}
\label{eq:malpos32}
[a^s_i+a^s_j,C_{ij}]=[b^s_i+b^s_j,C_{ij}]=0 \quad (\forall i\ne j\, , \forall s)
\end{equation}
\begin{equation}
\label{eq:malpos33}
[C_{ij}+C_{jk},C_{ik}]=0 \quad (\mbox{if }ij,ik,jk\in \se)
\end{equation}
\begin{equation}
\label{eq:malpos34}
[C_{ij},C_{jk}]=0 \quad (\mbox{if }ij,jk\in \se \mbox{ and \em ik}\notin \se)
\end{equation}

The first relation is obvious:
$$ [C_{ij},C_{kl}]=[C_{ij},[a_k^s,b_l^s]]=0\quad 
     (\mbox{by (\ref{eq:malpos21}) and (\ref{eq:malpos26})})\, . $$

The second equation comes from the equalities
$$\begin{array}{llll}
   [a^s_j,C_{ij}] & = & [a^s_j,\Sigma_kC_{ik}]        & (\mbox{by (\ref{eq:malpos26})}) \\
                  & = & [a^s_j,\Sigma_t[b_i^t,a_i^t]] & (\mbox{by (\ref{eq:malpos25})}) \\
                  & = & [a^s_j,[b_i^s,a_i^s]]         & (\mbox{by (\ref{eq:malpos23}) and \eqref{eq:malpos24}}) \\
                  & = & [C_{ij},a_i^s]                & (\mbox{by (\ref{eq:malpos21}) and \eqref{eq:malpos24}})  
\end{array}$$
(by symmetry, we get $[b^s_i+b^s_j,C_{ij}]=0$). 

Using (\ref{eq:malpos32}), we can finish the proof as follows:
$$\begin{array}{llll}
 [C_{ij}+C_{jk},C_{ik}] & = & [[a_i^s,b_j^s]+[a_k^s,b_j^s],C_{ik}] & 
                              (\mbox{by (\ref{eq:malpos21})}) \\
                        & = & [[a_i^s+a_k^s,b_j^s],C_{ik}]=0       & 
                              (\mbox{by (\ref{eq:malpos26}) and (\ref{eq:malpos32})}) \, ,  
\end{array}$$
and finally \eqref{eq:malpos34} may be established as follows:
$$\begin{array}{llll}
   [C_{ij},C_{jk}] & = & [C_{ij},[a_j^s,b_k^s]]       & (\mbox{by (\ref{eq:malpos21})}) \\
                   & = & [[C_{ij},a_j^s],b_k^s]       & (\mbox{by (\ref{eq:malpos26})}) \\
                   & = & -[[C_{ij},a_i^s],b_k^s]      & (\mbox{by (\ref{eq:malpos32})}) \\
                   & = & -[C_{ij},[a_i^s,b_k^s]]=0    & 
           (\mbox{by (\ref{eq:malpos26}), (\ref{eq:malpos21}) and (\ref{eq:malpos22})})\, .  
\end{array}$$  
\end{proof}

\begin{example}
\label{ex:filtvs1f}
Note that filtered formality is strictly weaker than $1$--formality, as shown by the 
Torelli group in genus $3$, which has cubic, non-$1$-formal Malcev Lie algebra, cf. Hain's 
work from \cite{H}.
\end{example}

\begin{prop}
\label{prop:malformal}
Suppose that either $g\ge 2$, or $g=1$ and $\G$ contains no $K_3$. Then the group $P(g, \G)$ is $1$--formal. 
\end{prop}

\begin{proof}
The cubic relations (\ref{eq:malpos26}) follow from the quadratic relations: if $g\geq 2$, 
take $t\ne s$; then
$$ [a^s_k,C_{ij}]=[a^s_k,[a_i^t,b_j^t]]=0 \quad (\mbox{by (\ref{eq:malpos21}), 
          (\ref{eq:malpos23}) and  (\ref{eq:malpos24})}); $$
if $g=1$ and, say, $ik\notin \se $, we find
$$ [a^1_k,C_{ij}]=[a^1_k,[a_j^1,b_i^1]]=0 \quad (\mbox{by (\ref{eq:malpos21}), 
          (\ref{eq:malpos22}) and  (\ref{eq:malpos24})}). $$        
\end{proof}

\begin{prop}
\label{prop:malnot1f}
If  $g=1$ and $\G$ contains a $K_3$ subgraph, then the group $P(1, \G)$ is not $1$--formal. 
\end{prop}

\begin{proof}
When $g\ge 1$ and $f: \G' \inj \G$ is arbitrary, note that 
$f_* : H_1(\Sigma_g^{\sv}) \surj H_1(\Sigma_g^{\sv'})$ extends to a graded Lie surjection,  
$f_* : \L^{\hdot}(H_1(\Sigma_g^{\sv})) \surj \L^{\hdot}(H_1(\Sigma_g^{\sv'}))$,
which preserves the graded parts of the defining Lie ideals \eqref{eq:malpos21}--
\eqref{eq:malpos26}. Furthermore, the canonical injection 
$f_{\dagger} : H_1(\Sigma_g^{\sv'}) \inj H_1(\Sigma_g^{\sv})$ extends to a graded Lie
monomorphism, 
$f_{\dagger} : \L^{\hdot}(H_1(\Sigma_g^{\sv'})) \inj \L^{\hdot}(H_1(\Sigma_g^{\sv}))$,
which preserves the cubic relations \eqref{eq:malpos26}. Therefore, the $1$--formality of 
$P(1, \G)$ would imply the $1$--formality of $P(1, K_3)$, in contradiction with 
\cite[Example 10.1]{DPS09}.
\end{proof}

\begin{remark}
\label{rk:notprod}
It follows from Proposition \ref{prop:res2} and \cite[Proposition 5.6]{M10} that, when 
$g\ge 2$, $\RR^1_1(A^{\hdot}(g, \G))= \RR^1_1(H^{\hdot}(\Sigma_g^{\sv}))$, for any graph 
$\G$. Nevertheless, $\m (P(g, \G)) \not\simeq \m (\pi_1(\Sigma_g^{\sv}))$, if 
$\se \ne \emptyset$. Indeed, assuming the contrary we infer from \cite{S} that the spaces 
$F(g, \G)$ and $\Sigma_g^{\sv}$ have isomorphic decomposable subspaces in the cohomology 
ring, in degree $2$: $DH^2 (F(g, \G)) \simeq DH^2 (\Sigma_g^{\sv})$. Plainly, 
$DH^2 (\Sigma_g^{\sv})= H^2 (\Sigma_g^{\sv})$. The description of the Orlik--Solomon model 
$A^{\hdot}(g, \G)$ from Section \ref{sec:pencils} readily implies that 
$DH^2 (F(g, \G))= H^2 (\Sigma_g^{\sv})/dG$. By Lemma \ref{lem:prel}\eqref{pr1}, the above 
two vector spaces $DH^2$ have different dimensions if $\se \ne \emptyset$, a contradiction.
\end{remark}


\section{Non-abelian representation varieties and jump loci}
\label{sec:sl2}

Finally, we analyze germs at $1$ of rank $2$ non-abelian representation varieties and their 
degree one topological Green--Lazarsfeld loci for partial pure braid groups, via admissible 
maps and Orlik--Solomon models, and we prove Theorem \ref{thm:main3}. In this section, $\GG =\SL_2(\C)$ 
or its standard Borel subgroup, with Lie algebra $\g=\sl_2$ or $\sol_2$. Key to our 
computations is the well-known fact that $[A,B]=0$ in $\g$ if and only if $\rank\{A,B\}\le 1$.

If $S=\overline{S} \setminus F$ is a quasi-projective curve, where $\overline{S}$ is projective 
and $F\subseteq \overline{S}$ is a finite subset, then $(\overline{S}, F)$ is the unique 
smooth compactification of $S$. For a quasi-projective manifold $M$, it is known that there 
is a {\em convenient} smooth compactification, $M=\overline{M} \setminus D$, where $D$ is a 
hypersurface arrangement in $\overline{M}$,  which has the property that every admissible map 
of general type, $f:M\to S$, is induced by a regular morphism, 
$\overline{f}: (\overline{M},  D) \to (\overline{S}, F)$. These in turn induce $\cdga$ maps 
between Orlik--Solomon models, denoted $f^*:A^{\hdot}(\overline{S},F)\to  A^{\hdot}(\overline{M},D)$. 
By naturality, we obtain an inclusion
\begin{equation}
\label{eq:flatpen}
\F(A^{\hdot}(\overline{M}, D), \g) \supseteq \F^1(A^{\hdot}(\overline{M}, D), \g) \cup 
\bigcup_{f\in \cE (M)} f^*\F(A^{\hdot}(\overline{S}, F), \g)\, .
\end{equation} 
For any finite-dimensional representation $\theta: \g \to \gl (V)$, we also know from 
\cite[Corollary 3.8]{MPPS} 
that $\Pi (A, \theta) \subseteq \RR^k_1 (A, \theta)$, if $H^k(A) \ne 0$. 

Let $\{ f: B^{\hdot}_f \to A^{\hdot} \}$ be a finite family of $\cdga$ maps between finite 
objects.

\begin{prop}
\label{prop:fltores}
Assume that $H^1(A) \ne 0$. For every $f$, suppose that $B^{\hdot}_f =B^{\le 2}_f $, 
$\chi (H^{\hdot}(B_f))<0$ and $f$ is a monomorphism. If $\RR^1_1 (A)= \bigcup_f \im H^1(f)$ and 
\eqref{eq:flatpen} holds as an equality for the family $\{f:B^{\hdot}_f\to A^{\hdot}\}$, then 
\begin{equation}
\label{eq:respen}
\RR^1_1 (A, \theta) = \Pi (A, \theta) \cup \bigcup_f f^* \F(B_f, \g)\, ,
\end{equation}
for any finite-dimensional representation $\theta: \g \to \gl (V)$.
\end{prop}

\begin{proof}
The inclusion "$\supseteq$". The fact that $\Pi (A, \theta) \subseteq \RR^1_1 (A, \theta)$
is due to the assumption $H^1(A) \ne 0$. The equality $\RR^1_1 (B_f, \theta)= \F (B_f, \g)$ 
follows from \cite[Proposition 2.4]{MPPS}, since $B^{\hdot}_f =B^{\le 2}_f $ and 
$\chi (H^{\hdot}(B_f))<0$. Lemma 2.6 from \cite{MPPS} implies that 
$f^* \RR^1_1 (B_f, \theta) \subseteq \RR^1_1 (A, \theta)$, since
$f$ is injective in degree $1$. To verify the inclusion "$\subseteq$", pick 
$\omega \in \RR^1_1 (A, \theta) \setminus \bigcup_f f^* \F(B_f, \g)$. We infer from 
\eqref{eq:flatpen} that $\omega =\eta \otimes g$, with $d\eta =0$ and $g\in \g$. 
Theorem 1.2 from \cite{MPPS} says then that there is an eigenvalue $\lambda$ of $\theta (g)$ 
such that $\lambda \eta \in \RR^1_1 (A)$. If $\det \theta (g)\ne 0$, then $\lambda\ne 0$. 
Since $\RR^1_1 (A)= \bigcup_f \im H^1(f)$, we deduce that $\eta =f^* \eta_f$, for
some $f$ and some $\eta_f \in H^1(B_f)$. Hence, $f^*(\eta_f \otimes g)\in \F(A, \g)$. The 
injectivity of $f$ forces then $\eta_f \otimes g \in \F(B_f, \g)$. This implies that 
$\omega\in f^*\F(B_f,\g)$, a contradiction. Consequently, $\omega\in\Pi(A,\theta)$, and 
we are done.
\end{proof}

Let $A$ be a finite model of the finite space $M$. If $b_1(M)=0$, then it follows from \cite{Q} 
that $\m (\pi_1(M))=0$. Theorems A and B in \cite{DP-ccm} together imply then that both germs 
$\Hom (\pi_1(M), \GG)_{(1)}$ and $\F(A, \g)_{(0)}$ contain only the origin. Furthermore, 
$b_1(M)=0$ implies that $\VV^1_1 (M, \iota)_{(1)}= \RR^1_1 (A, \theta)_{(0)}= \emptyset$, cf. 
\cite[Theorem B]{DP-ccm} and \cite[(15)]{MPPS}. For a quasi-projective manifold $M$ with 
$b_1(M)>0$, it follows from \cite[Example 5.3]{DP-ccm} that we may always find a 
convenient compactification (by adding at infinity a normal crossing divisor) which
satisfies all hypotheses from Proposition \ref{prop:fltores}, for the family 
$\{ f^*: A^{\hdot}(\overline{S}, F) \to  A^{\hdot}(\overline{M}, D)\}_{f\in \cE (M)}$, 
except possibly the last assumption.

In this way, we infer from Remark \ref{rk:w1} and Proposition \ref{prop:fltores} that the genus 
$0$ case of Theorem \ref{thm:main3} becomes a consequence of the following general result. 

\begin{theorem}
\label{thm:flat0}
If $b_1(M)>0$ and $W_1H^1(M)=0$, then equality holds in \eqref{eq:flatpen}, for
a convenient compactification with normal crossings and for $\g=\sl_2$ or $\sol_2$. 
\end{theorem}

\begin{proof}
For every $f\in \cE (M)$, note that 
$H^{\hdot}(\overline{f}): H^{\hdot}(\overline{S}) \to H^{\hdot}(\overline{M})$
is injective (see e.g. \cite[Example 5.3]{DP-ccm}). Our vanishing assumption on $W_1H^1(M)$ 
implies that $H^{1}(\overline{M})=0$, cf. \cite{De2, De3}. Hence, $W_1H^1(S)=0$.

Let $A^{\hdot}_{\hdot} :=A^{\hdot}(\overline{M}, D)$ be the Gysin model, and assume that 
$W_1H^1(M)=0$. Then $A^1=A^1_2$, by \cite{M}. Set $Z^1_2:= H^1(A)\subseteq A^1_2$, and denote 
by $A^{\hdot}_Z \subseteq A^{\le 2}$ the sub--$\cdga$ with $d=0$ defined by 
$A^0_Z= \Q \cdot 1$, $A^1_Z=Z^1_2$ and $A^2_Z= \mu (\bigwedge^2 Z^1_2) \subseteq A^2_4$. 
Note that $d(A^1_2)\subseteq A^2_2$. We infer that the $\cdga$ inclusion 
$\iota : A^{\hdot}_Z \inj A^{\le 2}$ is a $1$--equivalence, i.e., it induces an isomorphism 
in $H^1$ and a monomorphism in $H^2$. On the other hand, it follows from the definitions that 
the variety $\F (A, \g)$ depends only on the co-restrictions of $d:A^1 \to A^2$ 
and $\mu : \bigwedge^2 A^1 \to A^2$ to the subspace $\im (d) +\im (\mu)\subseteq A^2$, 
for any $\cdga$ $A$ and any Lie algebra $\g$. Therefore, we have an inclusion
$\iota^* : \F (A_Z, \g) \subseteq \F (A, \g)$.

Since $\iota$ is a $1$--equivalence, it follows from Theorem 3.9 and \S\S 7.3--7.5 in 
\cite{DP-ccm} that $\F (A_Z, \g)$ and $\F (A, \g)$ have the same analytic germs at $0$. Now, 
we recall from \cite{DP-ccm} that both $\cdga$'s, $A$ and $A_Z$, have positive weights, and 
the associated $\C^{\times}$-actions preserve the varieties $\F (A_Z, \g)$, $\F (A, \g)$ 
and the origin $0$. This implies that all irreducible components of $\F(A,\g)$ pass through 
$0$, and similarly for $\F (A_Z, \g)$. This in turn is enough to infer that actually 
$\F (A_Z, \g) = \F (A, \g)$, since the germs at $0$ are equal. Moreover, 
$\F (A_Z, \g)= \F (H^{\hdot}(A), \g)$, by construction.

The equalities $\F(A^{\hdot}(\overline{M}, D), \g)= \F(H^{\hdot}(M), \g)$ and 
$\F(A^{\hdot}(\overline{S}, F), \g)= \F(H^{\hdot}(S), \g)$ are clearly compatible with the 
natural maps induced by $\overline{f}: (\overline{M},  D) \to (\overline{S}, F)$, 
for any $f\in \cE (M)$. Plainly $\F^1 (A^{\hdot}(\overline{M}, D), \g)$ depends only on 
$H^1(M)$ and $\g$. Thus, we may replace in \eqref{eq:flatpen} $A^{\hdot}(\overline{M}, D)$ by 
$(H^{\hdot}(M), d=0)$ and $A^{\hdot}(\overline{S}, F)$ by $(H^{\hdot}(S), d=0)$. 
In this way, our claim reduces to the equality proved in \cite[Corollary 7.2(55)]{MPPS}.
\end{proof}

In positive genus, we are going to describe explicitly the convenient compactifications from 
Theorem \ref{thm:main3}, and check that all hypotheses from Proposition \ref{prop:fltores} 
hold for the associated families of $\cdga$ maps, 
$\{ f^*: A^{\hdot}(\overline{S}, F) \to  A^{\hdot}(\overline{M}, D)\}_{f\in \cE (M)}$,  except 
the last assumption.

When $g\ge 2$, $M:=F(g, \G)=\Sigma_g^{\sv} \setminus D_{\G}$ is a convenient compactification: 
for $i\in \sv$, the regular morphism 
$\overline{f_i} := \proj_i : (\Sigma_g^{\sv}, D_{\G}) \to (\Sigma_g, \emptyset)$ extends
the admissible map $f_i: F(g, \G) \to \Sigma_g$ from Lemma \ref{lem:somepen}. By Lemma 
\ref{lem:prel}\eqref{pr1}, $H^1(A(g, \G)) \ne 0$, for $g\ge 1$. Clearly,  
$B^{\hdot}_f =B^{\le 2}_f $ and $\chi (H^{\hdot}(B_f))<0$, for any
$f\in \cE (M)$, since $B^{\hdot}_f = (H^{\hdot}(\Sigma_g), d=0)$. It is easy to check that 
$f^*: A^{\hdot}(g, \G') \to A^{\hdot}(g, \G)$ is injective in degree $\hdot \le 2$, for any 
$f:\G' \inj \G$ and $g\ge 0$. Finally, the assumption on $\RR^1_1(A(g, \G))$ 
in Proposition \ref{prop:fltores} follows from Proposition \ref{prop:res2}.

In genus $g=1$, $M:=F(1, \G)=\Sigma_1^{\sv} \setminus D_{\G}$ is again a convenient 
compactification. For $ij\in \se$,
denote by $\proj_{ij} : (\Sigma_1^{\sv}, D_{\G}) \to (\Sigma_1^{2}, D_{K_2})$ the regular 
morphism induced by projection.
Let $\overline{\delta} : (\Sigma_1^{2}, D_{K_2}) \to  (\Sigma_1, \{ 0\})$ be the regular 
morphism induced by the difference map of the elliptic curve $\Sigma_1$. 
Then clearly the  regular morphism $\overline{f_{ij}} := \overline{\delta} \circ \proj_{ij}$
extends the admissible map $f_{ij}: F(1, \G) \to \Sigma_1 \setminus \{ 0\}$ from Lemma 
\ref{lem:somepen}.
For any $f\in \cE (M)$, $B^{\hdot}_f =A^{\hdot} (\Sigma_1, \{ 0\})=B^{\le 2}_f$ is given by 
$B^{0}_f =\C \cdot 1$, $B^{1}_f =\spn \{ x,y,g \}$ and $B^{2}_f =\C \cdot \OO$. 
The differential is given by $dx=dy=0$ and $dg=\OO$, and the multiplication table is 
$xg=yg=0$ and $xy=\OO$. The hypotheses on $B^{\hdot}_f$ from Proposition \ref{prop:fltores}
are clearly satisfied. It follows from naturality of Orlik--Solomon models \cite{D} that 
$\delta^*x=x_1-x_2$,  $\delta^*y=y_1-y_2$, and $\delta^*g=G_{12}$. In particular, 
$\delta^* : A^{\hdot}(\Sigma_1, \{ 0\}) \inj A^{\hdot}(1, K_2)$ is injective, which
proves the injectivity of $B^{\hdot}_f \to A^{\hdot}$ for any $f\in \cE (M)$. Finally, the 
assumption on $\RR^1_1(A(1, \G))$ in Proposition \ref{prop:fltores} follows 
from Proposition \ref{prop:res1}, when $\se \ne \emptyset$. Otherwise, the claims
in Theorem \ref{thm:main3} follow from \cite[Corollary 7.7]{MPPS}. 

By virtue of Proposition \ref{prop:fltores}, we have thus reduced the proof of Theorem 
\ref{thm:main3} in positive genus to checking that \eqref{eq:flatpen} holds as an equality 
for the families 
$\{ f^*: A^{\hdot}(\overline{S}, F) \to  A^{\hdot}(\overline{M}, D)\}_{f\in \cE (M)}$ described 
above. To verify this equality, we will use another basic property of the holonomy 
Lie algebra of a $\cdga$ $A$, proved in Proposition 4.5 from \cite{MPPS}.
This result allows us to naturally replace the variety of flat connections $\F (A, \g)$ by the 
variety of Lie homomorphisms, $\Hom_{\Lie} (\h (A), \g)$,   and $\F^1 (A, \g)$ by 
$\Hom_{\Lie}^1(\h(A),\g):=\{\varphi\in\Hom_{\Lie}(\h (A), \g) \mid \dim \im (\varphi)\le 1\}$.

\begin{prop}
\label{prop:flat1}
If $\varphi \in \Hom_{\Lie} (\h (A(1, \G)), \g) \setminus  \Hom_{\Lie}^1 (\h (A(1, \G)), \g)$, 
there is $ij\in\se$ such that $\varphi\in f_{ij}^*\Hom_{\Lie}(\h(A(\Sigma_1,\{0\})), \g)$.
\end{prop}

\begin{proof}
For $g\ge 1$, the holonomy Lie algebra $\h (A(g, \G))$ is isomorphic to the Lie algebra $L(g, \G)$
from Proposition \ref{prop:malpos}. By (\ref{eq:malpos24}), a morphism 
$\varphi \in \Hom_{\Lie} (\h (A(1, \G)), \g)$ satisfies
$$ [\varphi(a_i),\varphi(a_j)]=[\varphi(b_i),\varphi(b_j)]=0\, , $$
therefore $\varphi$ is defined by two elements $v,w\in\g$ and two $n$-vectors 
$\alpha_*=(\alpha_i)$, $\beta_*=(\beta_i)$:
$$ \varphi(a_i)=\alpha_iv,\, \varphi(b_i)=\beta_iw\, . $$
Equation (\ref{eq:malpos21}) implies that $(\alpha_i\beta_j-\alpha_j\beta_i)[v,w]=0$. If
$\varphi \notin \Hom_{\Lie}^1(\h (A(1, \G)), \g)$, we have $\alpha_*\ne 0$, 
$\beta_*\ne 0$ and $[v,w]\ne 0$, hence $\rank \{\alpha_*,\beta_*\}=1$. Equation 
(\ref{eq:malpos25}) is equivalent to 
$$ \Sigma_j[a_i,b_j]=\Sigma_j[a_j,b_i]=0\, \quad (i\in\sv)\, ;$$   
together with relation (\ref{eq:malpos24}), these imply that $\Sigma_ia_i,\Sigma_ib_i$ are 
central elements, therefore their images $\Sigma_i\alpha_iv$, $\Sigma_i\beta_iw$ are 0. 
In particular, at least two components of $\alpha_*$ (and the same components of $\beta_*$)
are non-zero. 

We will show that $\alpha_*$ and $\beta_*$ have exactly two 
non-zero components. Relations (\ref{eq:malpos21}) and (\ref{eq:malpos26}) imply that, 
for any three distinct indices $i,j,k$, 
$$\alpha_k\alpha_i\beta_j[v,[v,w]]=\beta_k\alpha_i\beta_j[w,[v,w]]=0\, .$$  
The two brackets $[v,[v,w]],[w,[v,w]]$ cannot be both 0 (otherwise $\rank\{v,w\}=1$); if
$[v,[v,w]]\ne 0$, we have (for any three indices) $\alpha_k\alpha_i\beta_j=0$, which proves 
our claim (similarly if $[w,[v,w]]\ne 0$).    

We infer that $\varphi$ must be of the form
\begin{equation}
\label{eq=philie}
\varphi(a_i)=\alpha v,\,\varphi(a_j)=-\alpha v,\,\varphi(a_k)=0,\,  
\varphi(b_i)=\beta w,\,\varphi(b_j)=-\beta w,\,\varphi(b_k)=0\, ,
\end{equation}
with $\alpha, \beta \ne 0$ (where $k\ne i,j$). Therefore, $ij\in \se$, by \eqref{eq:malpos22}.

The description of $A^{\hdot}(\Sigma_1, \{0\})$ implies, by a straightforward computation, that the Lie algebra
$\h (A(\Sigma_1, \{ 0\}))$ is the quotient of the free Lie algebra $\L (x^*, y^*, g^*)$ by the relation
$g^*+[x^*,y^*]=0$, where $\{ x^*, y^*, g^*\}$ is the basis dual to $\{ x,y,g\}$. Therefore,
$\h (A(\Sigma_1, \{ 0\}))= \L (x^*, y^*)$. Moreover, the description of the action of $\delta^*$
and $\proj_{ij}^*$ on Orlik--Solomon models implies, by taking duals, that the Lie homomorphism 
$f_{ij *}: \h (A(1, \G)) \to \h (A(\Sigma_1, \{ 0\}))$ sends $a_i$ to $x^*$, $a_j$ to $-x^*$, 
$b_i$ to $y^*$, $b_j$ to $-y^*$, and $a_k, b_k$ to $0$ for $k\ne i,j$; see \cite[Definition 4.2]{MPPS}. 

Define $\psi\in\Hom_{\Lie} (\h (A(\Sigma_1, \{ 0\})), \g)$ by $x^*\mapsto \alpha v$, $y^*\mapsto \beta w$.
By \eqref{eq=philie}, $\varphi = f_{ij}^* (\psi)$. 
\end{proof}

\begin{prop}
\label{prop:flat2}
Assume that $g\ge 2$. If $\varphi \in \Hom_{\Lie} (\h (A(g, \G)), \g) \setminus  \Hom_{\Lie}^1 
(\h (A(g, \G)), \g)$, there is $i\in \sv$ 
such that $\varphi \in f_{i}^* \Hom_{\Lie} (\h (A(\Sigma_g, \emptyset)), \g)$.
\end{prop}

\begin{proof}
The holonomy Lie algebra of $A(\Sigma_g, \emptyset)=A(g, K_1)$ is generated by the elements
$\{a^1,b^1,\ldots,a^g,b^g\}$ modulo the relation $\Sigma_s[a^s,b^s]=0$, hence a morphism 
$\psi\in\Hom_{\Lie} (\h (A(\Sigma_g,\emptyset)), \g)$ is defined by $2g$ elements
$v^1,w^1,\ldots,v^g,w^g\in \g$ satisfying the relation $\Sigma_s[v^s,w^s]=0$. 

It is sufficient to show
that, for $\varphi\in\Hom_{\Lie}(\h (A(g,\G)),\g)\setminus\Hom_{\Lie}^1(\h(A(g,\G)),\g)$,
there is an index $i$ such that, for any $j\ne i$ and any $t$, 
$\varphi(a_j^{\,t})=\varphi(b_j^{\,t})=0$: 
this implies, via (\ref{eq:malpos21}), that $\varphi(C_{jk})=0$ (for any $j\ne k$) and, using 
(\ref{eq:malpos25}), that $\Sigma_s[\varphi(a_i^s),\varphi(b_i^s)]=0$.

Denote by $A$ and $B$ the span of $\{\varphi(a_*^*)\}$ and $\{\varphi(b_*^*)\}$ respectively. 
As $\dim\im(\varphi)\geq 2$, we have to analyze only two cases:
  
Case 1: $\dim (A)=\dim (B)=1$. In this case there are two linearly independent 
elements $v,w\in\g$ and indices $(i,s)$, $(k,t)$ such that
$$ \varphi(a_j^r)=\alpha_j^{\,r}v,\,\varphi(b_j^r)=\beta_j^{\,r}w,\,\mbox{ for any }j,r
     \mbox{ and } \alpha_i^s\ne 0\ne \beta_k^{\,t}\,.$$
Relation (\ref{eq:malpos23}) and $[v,w]\ne 0$ imply that $\beta_j^{\,r}=0$ if $j\ne i$ and 
$r\ne s$; from the hypothesis $g \geq 2$ and relation (\ref{eq:malpos21}), we obtain 
$$ \varphi(C_{ij})=\alpha_i^s\beta_j^s[v,w]=\alpha_i^{\,r}\beta_j^{\,r}[v,w]=0\, , $$ 
hence $\beta_j^{\,r}=0$ for any $j\ne i$ and any $r$. This implies that $k=i$ and, by symmetry,
that $\alpha_j^{\,r}=0$ for any $j\ne i$ and any $r$. 

Case 2: $\dim (A)\geq 2$ (by symmetry, the case $\dim (B)\geq 2$ can be treated in the same way). 
In this case there are indices $i=j$, $s\ne t$  and two linearly independent elements 
$v^s,v^{\,t}\in\g$ such that
$$ \varphi(a_i^s)=v^s,\,\varphi(a_j^t)=v^{\,t} $$
($i\ne j$ contradicts relation (\ref{eq:malpos24}), since $[v^s,v^{\,t}]\ne 0$). For 
any $k\ne i$ and any $r$, we obtain from (\ref{eq:malpos24}) that 
$$ [\varphi(a_i^s),\varphi(a_k^r)]=[\varphi(a_i^t),\varphi(a_k^r)]=0\, , \mbox{ hence } 
    \varphi(a_k^r)=0\, . $$
Using relation (\ref{eq:malpos23}), the same argument applied to $b_k^r$ shows that 
$\varphi(b_k^r)=0$ for any $k\ne i$ and any $r\ne s,t$. 
Again from (\ref{eq:malpos23}), $[\varphi (a_i^t), \varphi (b_k^s)]=0$. On the other hand, by
(\ref{eq:malpos21}), $[\varphi (a_i^s), \varphi (b_k^s)]=[\varphi (a_k^t), \varphi (b_i^t)]=0$. 
Hence, $\varphi(b_k^r)=0$ for any $k\ne i$ and $r= s,t$, and we are done.
\end{proof}

Propositions \ref{prop:flat1} and \ref{prop:flat2} complete the proof of Theorem \ref{thm:main3}. 
Similar results were obtained in 
\cite{MPPS}, for quasi-projective manifolds with $1$--formal fundamental group. 
(Note that $(H^{\hdot}(S), d=0)$ is a finite model of a quasi-projective curve $S$, and 
$\F ((H^{\hdot}(S), d=0), \g)$ is computed in Lemma 7.3 from \cite{MPPS}, when $\chi (S)<0$.) 
They were based on the following algebraic construction. Let $A^{\hdot}$ 
be a $1$--finite $\cdga$ with linear resonance, i.e., $\RR^1_1 (A)= \bigcup_{C\in \cC} C$ 
is a finite union of linear subspaces of $H^1(A)$. For each $C\in \cC$, let
$A^{\hdot}_C \inj A^{\le 2}$ be the sub-$\cdga$ defined by $A^{0}_C =\C \cdot 1$,  $A^{1}_C =C$ 
and $A^{2}_C =A^2$.

\begin{prop}[\cite{MPPS}, Proposition 5.3]
\label{prop:flatalg}
If in addition $d=0$ then 
\[
\F (A, \g)= \F^1 (A, \g) \cup \bigcup_{C\in \cC} \F (A_C, \g)\, ,
\]
for $\g =\sl_2$ or $\sol_2$.
\end{prop}

\begin{example}
\label{ex:notalg}
The geometric formulae from Theorem \ref{thm:main3}, based on Orlik--Solomon models, seem to be the 
right extension of the similar results in \cite{MPPS}, beyond the 
$1$-formal case. Indeed, let us consider for $A^{\hdot}= A^{\hdot}(1, \G)$ the linear 
decomposition of $\RR^1_1 A)$ from Proposition \ref{prop:res1}, case $\se\ne\emptyset$. We 
claim that, for each $C=\im H^1(f_{ij})$, $\F (A_C, \g)= \F^1 (A_C, \g)$, when $\g =\sl_2$ or 
$\sol_2$. This implies that the algebraic formula from Proposition \ref{prop:flatalg} reduces 
in this case to the equality $\F (A, \g)= \F^1 (A, \g)$. On the other hand, we have seen that 
$\h (A(\Sigma_1, \{ 0\}))$ is a free Lie algebra on two generators, and therefore 
$\F (A(\Sigma_1, \{ 0\}), \g)$ contains an element not in $\F^1 (A(\Sigma_1, \{ 0\}), \g)$. 
Consequently, if $ij \in \se$ then it follows from Theorem \ref{thm:main3} that 
$f_{ij}^* \F (A(\Sigma_1, \{ 0\}), \g) \setminus \F^1 (A(1, \G), \g) \ne \emptyset$. Thus, the 
algebraic formula does not hold.

To compute $\h (A_C)$, we may replace $A^2_C$ by $\mu_C (\bigwedge^2 C)$. Note that $d_C=0$, 
$C$ is two-dimensional generated by $x_i-x_j$ and  $y_i-y_j$, and $(x_i-x_j)(y_i-y_j)\ne 0$. 
It follows that the holonomy Lie algebra $\h (A_C)$ is two-dimensional abelian. Therefore, 
$\Hom_{\Lie} (\h (A_C), \g)= \Hom_{\Lie}^1 (\h (A_C), \g)$, as claimed. 
\end{example}

\newcommand{\arxiv}[1]
{\texttt{\href{http://arxiv.org/abs/#1}{arxiv:#1}}}
\newcommand{\arx}[1]
{\texttt{\href{http://arxiv.org/abs/#1}{arXiv:}}
\texttt{\href{http://arxiv.org/abs/#1}{#1}}}
\newcommand{\doi}[1]
{\texttt{\href{http://dx.doi.org/#1}{doi:#1}}}


\begin{thebibliography}{00}

\bibitem{Ad} J.~F.~Adams,
{\em On the cobar construction}, in {\em Colloque de Topologie Alg\' ebrique}
(Louvain, 1956), George Thone, Liege, Masson, Paris, 1957, pp.~81--87.

\bibitem{A}  D.~Arapura, 
{\em Geometry of cohomology support loci for local systems {\rm I}}, 
J. Algebraic Geom. \textbf{6} (1997), no.~3, 563--597.  

\bibitem{AA}  S.~Ashraf, H.~Azam, B.~Berceanu, 
{\em Representation theory for the Kriz model}, 
Algebraic and Geometric Topology \textbf{14} (2014), 57--90.  

\bibitem{BP} B.~Berceanu, S.~Papadima,
{\em Universal representations of braid and braid--permutation groups},
J. Knot Theory Ramif. \textbf{18} (2009), no.~7, 999--1019.

\bibitem{B} R.~Bezrukavnikov,
{\em Koszul dg--algebras arising from configuration spaces},
Geometric and Functional Analysis \textbf{4} (1994), no.~2, 119--135.

\bibitem{BH} C.~Bibby, J.~Hilburn,
{\em Quadratic--linear duality and rational homotopy theory of chordal arrangements},
Algebraic and Geometric Topology \textbf{16} (2016), 2637–-2661.

\bibitem{Ch} K.--T.~Chen, 
{\em Iterated path integrals}, Bull. Amer. Math. Soc. \textbf{83} (1977), no.~5, 831--879.

\bibitem{De2}  P.~Deligne,
{\em Th\'{e}orie de Hodge. {\rm II}},
Inst. Hautes \'{E}tudes Sci. Publ. Math. \textbf{40}
(1971), 5--57.

\bibitem{De3}  P.~Deligne,
{\em Th\'{e}orie de Hodge. {\rm III}},
Inst. Hautes \'{E}tudes Sci. Publ. Math. \textbf{44} (1974), 5--77.

\bibitem{D10} A.~Dimca, 
{\em Characteristic varieties and logarithmic differential 1-forms}, 
Compositio Math. \textbf{146} (2010),  129--144.

\bibitem{DP-ccm}  A.~Dimca, S.~Papadima,
{\em Non-abelian cohomology jump loci from an analytic 
viewpoint}, Commun. Contemp. Math. \textbf{16} (2014), 
no.~4, 1350025 (47 p). 

\bibitem{DPS09}  A.~Dimca, S.~Papadima, A.~Suciu,
{\em Topology and geometry of cohomology jump loci}, 
Duke Math. Journal \textbf{148} (2009), no.~3, 405--457.

\bibitem{D} C.~Dupont,
{\em The Orlik--Solomon model for hypersurface arrangements},
Ann. Inst. Fourier \textbf{65} (2015), no.~6, 2507--2545.

\bibitem{GL} M.~Green, R.~Lazarsfeld,
{\em Deformation theory, generic vanishing theorems, and some conjectures of
Enriques, Catanese and Beauville}, Invent. Math. \textbf{90} (1987), 389--407.

\bibitem{Ha1} R.~Hain,
{\em The de Rham homotopy theory of complex algebraic varieties {\rm I}},
K--Theory \textbf{1} (1987), 271--324.

\bibitem{H} R.~Hain,
{\em Infinitesimal presentations of the {T}orelli groups},
J. Amer. Math. Soc. \textbf{10} (1997), no.~3, 597--651.

\bibitem{M10}  A.~Macinic,
{\em Cohomology rings and formality properties of nilpotent groups},
J. Pure Appl. Algebra \textbf{214} (2010), 1818--1826.

\bibitem{MPPS}  A.~M\u{a}cinic, S.~Papadima, R.~Popescu, 
A.~Suciu, {\em Flat connections and resonance varieties: 
from rank one to higher ranks}, Trans. Amer. Math. Soc. 
\textbf{369} (2017), 1309--1343.


\bibitem{M}  J.~W.~Morgan,
{\em The algebraic topology of smooth algebraic varieties},
Inst. Hautes \'{E}tudes Sci. Publ. Math. \textbf{48} (1978), 137--204.

\bibitem{OT} P.~Orlik, H.~Terao,
{\em Arrangements of hyperplanes}, Grundlehren Math. Wiss., vol.~300,
Springer-Verlag, Berlin, 1992.


\bibitem{Q}  D.~Quillen,
{\em Rational homotopy theory}, Annals of Math. 
\textbf{90} (1969), 205--295.


\bibitem{SS} H.~Schenck, A.~Suciu, 
{\em Resonance, linear syzygies, Chen groups, and the 
Bernstein--Gelfand--Gelfand correspondence}, 
Trans. Amer. Math. Soc. \textbf{358} (2005), 
no.~5, 2269--2289.

\bibitem{S}  D.~Sullivan,
{\em Infinitesimal computations in topology}, Inst. Hautes 
\'{E}tudes Sci. Publ. Math. \textbf{47} (1977), 269--331.


\end{thebibliography}
\end{document}